\documentclass[10pt,reqno]{amsart}

\usepackage{color}

\newtheorem{theorem}{Theorem}[section]
\newtheorem{lemma}[theorem]{Lemma}
\newtheorem{corollary}[theorem]{Corollary}

\theoremstyle{remark}
\newtheorem{remark}[theorem]{Remark}

\theoremstyle{definition}
\newtheorem{assumption}[theorem]{Assumption}
\newtheorem{example}[theorem]{Example}
\newtheorem{definition}[theorem]{Definition}

\def\qed{{\hfill $\Box$ \bigskip}}
\def\loc{{\rm loc}}

\def\XXint#1#2#3{{\setbox0=\hbox{$#1{#2#3}{\int}$}
\vcenter{\hbox{$#2#3$}}\kern-.5\wd0}}

\newcommand\aint{-\hspace{-0.38cm}\int}

\newcommand\cbrk{\text{$]$\kern-.15em$]$}}
\newcommand\opar{\text{\,\raise.2ex\hbox{${\scriptstyle
|}$}\kern-.34em$($}}
\newcommand\cpar{\text{$)$\kern-.34em\raise.2ex\hbox{${\scriptstyle |}$}}\,}

\def\<{\langle}
\def\>{\rangle}

\newcommand\bR{\mathbb{R}}

\newcommand\bS{\mathbb{S}}

\newcommand\bE{\mathbb{E}}

\newcommand\cI{\mathcal{I}}
\newcommand\cA{\mathcal{A}}
\newcommand\cB{\mathcal{B}}
\newcommand\cC{\mathcal{C}}

\newcommand\cF{\mathcal{F}}

\newcommand\cH{\mathcal{H}}
\newcommand\cJ{\mathcal{J}}
\newcommand\cK{\mathcal{K}}

\newcommand\cM{\mathcal{M}}

\newcommand\osc{\text{osc}}

\def\R {{\mathbb R}}

\def\Re{\text{Re}}
\def\Im{\text{Im}}

\begin{document}

\title[An $L_p$-theory for  non-local elliptic equations]
{An $L_p$-theory for a class of  non-local elliptic equations related to   nonsymmetric measurable kernels}

\author{Ildoo Kim}
\address{Department of Mathematics, Korea University, 1 Anam-dong, Sungbuk-gu, Seoul,
136-701, Republic of Korea} \email{waldoo@korea.ac.kr}

\author{Kyeong-Hun Kim}
\address{Department of Mathematics, Korea University, 1 Anam-dong,
Sungbuk-gu, Seoul, 136-701, Republic of Korea}
\email{kyeonghun@korea.ac.kr}
\thanks{The research of the second
author was supported by Basic Science Research Program through the
National Research Foundation of Korea(NRF) funded by the Ministry of
Education, Science and Technology (2013020522)}



\begin{abstract}
We study the integro-differential operators $L$ with kernels $K(y) = a(y) J(y)$, where $J(y)dy$ is a L\'evy measure on $\bR^d$ (i.e. $\int_{\bR^d}(1\wedge |y|^2)J(y)dy<\infty$) and $a(y)$ is an only measurable function with positive lower and upper bounds.
Under few additional conditions on $J(y)$, we prove the  unique solvability of the equation $Lu-\lambda u=f$ in $L_p$-spaces and present some $L_p$-estimates of the solutions.

\end{abstract}

\maketitle

\section{introduction}
There has been growing interest in the  integro-differential equations related to pure jump processes owing to their applications in  various models in physics, biology, economics, engineering and many others involving long-range jumps and interactions. In this article we  study the non-local elliptic equations having  the operators
$$
Lu := \int_{\R^d}\Big(u(x+y)-u(x)- y \cdot \nabla u(x) \chi (y) \Big)\, K(x,y) dy,
$$
and
$$
\tilde{L}u:=\int_{\R^d}\Big(u(x+y)-u(x)- y \cdot \nabla u(x) 1_{|y|<1} \Big)\, K(x,y) dy,
$$
where the kernel $K(x,y)=a(y)J(y)$ depends only on $y$,
\begin{eqnarray*}
\chi (y) = 0 ~\text{if}~ \sigma \in (0,1), \quad \chi (y) = 1_{|y|< 1} ~\text{if}~ \sigma = 1,\quad \chi (y) = 1 ~\text{if}~ \sigma \in (1,2].
\end{eqnarray*}
The constant $\sigma$  depends on $J(y)$ and  is defined in (\ref{sigma}). In particular, if $J(y)=c(d,\alpha)|y|^{-d-\alpha}$ for some $\alpha\in (0,2)$ then $\sigma=\alpha$. Note  that  if $a(y)$ is symmetric then $\tilde{L}=L$, and in general we (formally) have
$$
\tilde{L}u=Lu+b\cdot \nabla u,
$$
where
$$ b^i=-\int_{B_1}y^i a(y)J(y)dy \quad \text{if}\,\,\sigma\in (0,1), \quad \quad  b^i=\int_{\bR^d\setminus B_1}y^i a(y)J(y)dy \quad \text{if}\,\,\sigma\in (1,2].
$$
The main goal  of this article is to prove the unique solvability of the equations
\begin{eqnarray}
                                          \label{inteq}
Lu - \lambda u = f \quad \text{and} \quad \tilde{L}u - \lambda u = f,\quad \lambda>0
\end{eqnarray}
in appropriate $L_p$-spaces and present some  $L_p$-estimates of the solutions. Here $p>1$.  If $p=2$, the only condition we are assuming is that $a(y)$ has positive lower and upper bounds and   $J(y)$ is  rotationally invariant.
If $p\neq  2$, we  assume some
additional conditions  on $J(y)$, which are described in (\ref{eqn crucial 11}) and (\ref{eqn crucial 12}) below (also see Assumption \ref{assump 10.25.3}).

Below is a short description on related $L_p$-theories. For other results such as  the Harnack inequality and H\"older estimates  we refer the readers to \cite{bass1}, \cite{bass2}, \cite{dkim2}, \cite{KSV8} and \cite{S}.
 If $K(x,y)= c(d,\alpha) |y|^{-d-\alpha}$,  where $\alpha \in (0,2)$ and $c(d,\alpha)$ is some normalization constant, then  $L$ becomes the fractional Laplacian operator $\Delta^{\alpha/2}:=-(-\Delta)^{\alpha/2}$.
 For the fractional Laplacian operator, $L_p$-estimates can be easily  obtained
 by  the Fourier multiplier theory (for instance, \cite{Ste2}).
In  \cite{bb2}  $L_p$-estimates  were obtained for  elliptic equations  with  ``symmetric" kernels, and  an $L_p$-theory for the equation $Lu-\lambda u=f$ with measurable nonsymmetric kernel $K(x,y)= a(y)|y|^{-d-\alpha}$ was  recently introduced in  \cite{dkim}.
 For parabolic equations,  the authors of \cite{MP2} handled the equations with the kernel $K(x,y)=a(x,y)|y|^{-d-\alpha}$ under the condition that the coefficient $a(x,y)$ is  homogeneous of order zero in $y$ and  sufficiently smooth in $y$, but it is allowed that $a$ also depends on $x$.
Lately in \cite{Zh}, an $L_p$-regularity theory for parabolic equations was constructed for  $J(y)$ satisfying
$$
\nu_1^\alpha (B) \leq \int I_B(y)J(y)dy \leq \nu_2^\alpha(B) \quad \forall B \in \cB(\bR^d),  
$$ 
where $\nu_i^{(\alpha)}$ are  L\'evy measures taking the form
\begin{eqnarray}
\label{130913}\nu_i^{(\alpha)}(B) := \int_{\bS^{d-1}} \Big( \int_0^\infty \frac{1_{B}(r\theta)dr}{r^{1+\alpha}} \Big) S_i(d\theta),
\end{eqnarray}
with finite surface measures $dS_i$  on  $\bS^{d-1}$. Since the same constant $\alpha$ is used for both $\nu_1^{(\alpha)}$ and $\nu_2^{(\alpha)}$, even the L\'evy measure $J(y)$ related to the operator $\Delta^{\alpha_1/2}+\Delta^{\alpha_2/2}$ is not of type (\ref{130913}) if $\alpha_1 \neq \alpha_2$.

From the probabilistic point of view, the fractional Laplacian operator can be described as  the infinitesimal generator of $\alpha$-stable processes. That is,   
$$
\Delta^{\alpha/2} f(x) = \lim_{t \to 0^+} \frac{1}{t}\bE [f(x+ X_t)-f(x)], \quad f\in C^{\infty}_0
$$
where $X_t$ is  an  $\bR^d$-valued L\'evy process in a probability space $(\Omega,P)$ with the characteristic function
 $\bE e^{i\lambda \cdot X_t}:=\int_{\Omega}e^{i\lambda \cdot X_t}\,dP=e^{-t|\lambda|^{\alpha}}$. More generally, for any Bernstein function $\phi$ with $\phi(0+)=0$ (equivalently, $\phi(\lambda)=\int^{\infty}_0(1-e^{-\lambda t})\mu(dt)$ for some  measure $\mu$ satisfying $\int^{\infty}_0 (1\wedge 1)\mu(dt)<\infty$), the operator $\phi (\Delta)$ is the infinitesimal generator of the process $X_t:=W_{S_t}$, where  $S_t$ is a subordinator (i.e. an increasing L\'evy process satisfying $S_0=0$)  with Laplace exponent $\phi$ (i.e. $\bE e^{\lambda S_t}=\exp\{t\phi(\lambda)\}$) and  $W_t$ is a $d$-dimensional Brwonian motion independent of $S_t$. Such process is called  the subordinate Brownian motion. Actually $\phi$ is a Bernstein function with $\phi(0+)=0$ if and only if it is a Laplace exponent of a subordinator. Furthermore, the relation
\begin{eqnarray}
                             \label{eqn 10.25.1}
\phi(\Delta) f :=- \phi(-\Delta) f
= \int_{\bR^d} \left(f(x+y) -f(x) -\nabla f (x) \cdot y \chi(y) \right)J(y)~dy
\end{eqnarray}
holds  with $j(|y|):=J(y)$ given by
\begin{eqnarray}
                  \label{jum}
j(r) = \int_0^\infty (4\pi t)^{-d/2} e^{-r^2/(4t)}~\mu(dt).
\end{eqnarray}
For the equations with the kernel $K(x,y) =a(y)J(y)$, an $L_p$-estimate  is obtained in  aforementioned article \cite{bb2} if $a(y)$ is symmetric.
 However to the best of our  knowledge,  if the coefficient $a(y)$ is only measurable and  $J(y) \neq |y|^{-d-\alpha}$ then the $L_p$-estimate has not been known yet. In this article we extend \cite{dkim}  to the class of  L\'evy measures $J(y)$ satisfying  the following two conditions: (i) there exists a constant $\alpha_0$, where $\alpha_0\in (0,1]$ if $\sigma\leq 1$ and $\alpha_0\in (1,2)$ if $\sigma>1$, so that
\begin{equation}
                \label{eqn crucial 11}
\frac{j(t)}{j(s)}\leq N (\frac{s}{t})^{d+\alpha_0}, \quad \forall \, 0<s \leq t,
\end{equation}
and, (ii) for any $t>0$
\begin{equation}
                \label{eqn crucial 12}
1_{\sigma<1}\int_{|y|\leq 1} |y| j(t |y|) ~dy + 1_{\sigma \geq 1}
\int_{|y|\leq 1} |y|^{2 } j(t |y|) ~dz \leq N j(t).
\end{equation}
See Section \ref{s:1} for few remarks on these conditions. 
It is easy to check that   (\ref{eqn crucial 11}) and (\ref{eqn crucial 12}) are satisfied if there  exists $\alpha\geq \alpha_0$ so that
\begin{equation}
                           \label{eqn crucial 13}
 (s/t)^{d+\alpha} j(s)\leq N_1 j(t) \leq N_2 (s/t)^{d+\alpha_0} j(s), \quad \forall \,\,\,0<s\leq t.
\end{equation}
One can  construct many interesting jump functions $j(t)$ satisfying (\ref{eqn crucial 13}).   For example, (\ref{eqn crucial 13}) holds if $J(y)$ is defined from (\ref{eqn 10.25.1}) and  $\phi$ is one of the following  (see Example \ref{exam 1}  for details):
\begin{itemize}
\item[(1)] $\phi(\lambda)=\sum_{i=1}^n\lambda^{\alpha_i}$, $0<\alpha_i<1$;
\item[(2)] $\phi(\lambda)=(\lambda+\lambda^\alpha)^\beta$, $\alpha, \beta\in (0, 1)$;
\item[(3)] $\phi(\lambda)=\lambda^\alpha(\log(1+\lambda))^\beta$, $\alpha\in (0, 1)$,
$\beta\in (0, 1-\alpha)$;
\item[(4)] $\phi(\lambda)=\lambda^\alpha(\log(1+\lambda))^{-\beta}$, $\alpha\in (0, 1)$,
$\beta\in (0, \alpha)$;
\item[(5)] $\phi(\lambda)=(\log(\cosh(\sqrt{\lambda})))^\alpha$, $\alpha\in (0, 1)$;
\item[(6)] $\phi(\lambda)=(\log(\sinh(\sqrt{\lambda}))-\log\sqrt{\lambda})^\alpha$, $\alpha\in (0, 1)$.
\end{itemize}
 In these cases, the jump function $j(r)$ is comparable to $r^{-d}\phi(r^{-2})$.

Our approach is borrowed from \cite{dkim}.
We estimate the sharp functions of the solutions and apply the Hardy-Littlewod  theorem and the Fefferman-Stein theorem. This approach is typically used to treat the second-order PDEs with small BMO or VMO coefficients (for instance, see \cite{Krbook}). In \cite{dkim} this method is applied  to a non-local operator with the kernel $K(x,y)=a(y)|y|^{-d-\alpha}$.
As in \cite{dkim}, our sharp function estimates are based on  some H\"older estimates of solutions. The original idea of obtaining H\"older estimates is   from \cite{bci}. Nonetheless, since we are considering much general $J(y)$ rather then $c(d,\alpha)|y|^{-d-\alpha}$, many new difficulties arise. In particular, our operators do not have the nice scaling property which is used in \cite{Krbook} and \cite{dkim}, and this cause many difficulties  in the estimates.

The article is organized as follows. In Section \ref{s:1}
 we introduce the  main results. Section \ref{s:2} contains the unique solvability in the $L_2$-space. In Section \ref{s:3} we establish some H\"older estimates of solutions.  Using these estimates we obtain the sharp function and maximal function estimates in Section \ref{s:4}. In Section \ref{s:5}, the proofs of main results are given.

We finish the introduction with some notation. As usual $\bR^{d}$ stands for the Euclidean space of points
$x=(x^{1},...,x^{d})$,  $B_r(x) := \{ y\in \bR^d : |x-y| < r\}$  and
$B_r
 :=B_r(0)$.
 For $i=1,...,d$, multi-indices $\beta=(\beta_{1},...,\beta_{d})$,
$\beta_{i}\in\{0,1,2,...\}$, and functions $u(x)$ we set
$$
u_{x^{i}}=\frac{\partial u}{\partial x^{i}}=D_{i}u,\quad
D^{\beta}u=D_{1}^{\beta_{1}}\cdot...\cdot D^{\beta_{d}}_{d}u,
\quad|\beta|=\beta_{1}+...+\beta_{d}.
$$
For an open set $U \subset \bR^d$ and a nonnegative non-integer constant $\gamma$, by $C^{\gamma}(U)$  we denote the usual H\"older space. For a nonnegative integer $n$, we write $u \in C^n(U)$  if $u$ is $n$-times continuously differentiable in $U$. By $C^n_0(U)$ (resp. $C^{\infty}_0(U)$) we denote the set of all functions in $C^n(U)$ (resp.   $C^{\infty}(U)$) with compact supports. Similarly by $C^n_b(U)$ (resp. $C^{\infty}_b(U)$) we denote the set of  functions in $C^n(U)$ (resp. $C^{\infty}(U)$) with bounded derivatives.
The standard $L_p$-space on $U$ with Lebesgue measure is denoted by $L_p(U)$.  We simply use $L_p$, $C^n$, $C_b^n$, $C_0^n$, $C_b^\infty$, and $C_0^\infty$ when $U=\bR^d$.  We use  ``$:=$" to denote a definition.  $a \wedge b = \min \{ a, b\}$ and $a \vee b = \max \{ a, b\}$. If we write $N=N(a, \ldots, z )$,
this means that the constant $N$ depends only on $a, \ldots , z$. The constant $N$ may change  from location to location, even within a line. By $\cF$ and $\cF^{-1}$ we denote the Fourier transform and the inverse Fourier transform, respectively. That is,
$\cF(f)(\xi) := \int_{\bR^{d}} e^{-i x \cdot \xi} f(x) dx$ and $\cF^{-1}(f)(x) := \frac{1}{(2\pi)^d}\int_{\bR^{d}} e^{ i\xi \cdot x} f(\xi) d\xi$.
For a Borel
set $A\subset \bR^d$, we use $|A|$ to denote its Lebesgue
measure and by $I_A(x)$ we denote  the indicator of $A$.

\vspace{1mm}

\section{Setting and main results} \label{s:1}

Throughout this article, we assume that $J(y)$ is rotationally invariant,
\begin{equation}
               \label{elliptic con}
\nu\leq a(y) \leq \Lambda
\end{equation}
 for some constants $\nu, \Lambda>0$, and
\begin{equation}
                    \label{levy measure}
 \int_{\bR^d}(1 \wedge |y|^2) J(y)~dy   <\infty.
\end{equation}
Let $e_1$ be a unit vector.  Obviously, the condition that $J(y)$ is rotationally invariant can be replaced by the condition that  $J(y)$ is comparable to $j(|y|):=J(|y|e_1)$, because $J(y)a(y)=j(|y|) \cdot a(y)J(y)j^{-1}(|y|):=j(|y|) \tilde{a}(y)$ and $\tilde{a}$ also has positive lower and upper bounds.

Denote
\begin{equation}
                         \label{sigma}
\sigma := \inf \{ \delta >0 : \int_{|y|\leq 1}\,|y|^\delta J(y)~dy   <\infty\},
\end{equation}
\begin{eqnarray*}
\chi (y) = 0 ~\text{if}~ \sigma \in (0,1), \quad \chi (y) = 1_{B_1} ~\text{if}~ \sigma = 1,\quad \chi (y) = 1 ~\text{if}~ \sigma \in (1,2].
\end{eqnarray*}
Note that if $J(y)=c(d,\alpha)|y|^{-d-\alpha}$ for some $\alpha\in (0,2)$ then we have $\sigma=\alpha$.

For $u \in C^2_b$ we introduce the non-local elliptic operators
\begin{eqnarray*}
\cA u= \int_{\R^d}\big(u(x+y)-u(x)- y \cdot \nabla u(x) \chi (y) \big)\, J(y)~ dy,
\end{eqnarray*}
\begin{eqnarray*}
Lu
= \int_{\R^d}\big(u(x+y)-u(x)-y \cdot \nabla u(x) \chi (y) \big)\, a(y)J(y)~ dy,
\end{eqnarray*}
$$
\tilde L u = \int_{\R^d}\big(u(x+y)-u(x)-y \cdot \nabla u(x) I_{B_1}(y) \big)\, a(y)J(y)~ dy,
$$
\begin{eqnarray*}
L^\ast u=\int_{\R^d}\big(u(x+y)-u(x)-y \cdot \nabla u(x) \chi (y) \big)\, a(-y)J(-y)~ dy,
\end{eqnarray*}
and
$$
\tilde L^\ast u = \int_{\R^d}\big(u(x+y)-u(x)-y \cdot \nabla u(x) I_{B_1}(y) \big)\, a(-y)J(-y)~ dy.
$$
We start with a simple but interesting result, which will be used later in the proof of Theorem \ref{maint2}.

\begin{lemma}
                    \label{lemma 5.25.1}
For any $p>1$ and $\lambda>0$,
$$
\|u\|_{L_p} \leq \frac{1}{\lambda} \|\tilde{L}u - \lambda u\|_{L_p}, \quad \forall\, u\in C^{\infty}_0.
$$
\end{lemma}
\begin{proof}
Put
$$
\Phi (\xi) := -\int_{\bR^d} (e^{i \xi  \cdot y} -1 - i (y \cdot \xi )I_{B_1} ) a(-y)J(-y)~dy
$$
and
$$
f:=\tilde L u -\lambda u.
$$
Since $a(-y)J(-y)$ is a L\'evy measure (i.e. $ \int_{\bR^d}(1 \wedge |y|^2) a(-y)J(-y)~dy   <\infty$), there exists a L\'evy process  whose characteristic exponent is $-t \Phi(\xi)$ (for instance, see Corollary 1.4.6 of \cite{apple}).
Denoting by $p_\Phi(t,dx)$ its law at $t$, we have
\begin{equation}
                    \label{eqn 5.28.1}
\int_{\bR^d} e^{-i \xi \cdot x} p_\Phi(t,dx)=\int_{\bR^d} e^{i (-\xi) \cdot x} p_\Phi(t,dx) =  e^{-t\Phi(-\xi)}.
\end{equation}
In non-probabilistic terminology it can be rephrased that  if  $ \int_{\bR^d}(1 \wedge |y|^2) a(-y)J(-y)~dy   <\infty$ then there exists a  continuous  measure-valued  function $p_{\Phi}(t,dx)$ such that $p_{\Phi}(t,\bR^d)=1$ and  (\ref{eqn 5.28.1}) holds.
Since
$$
(-\Phi(-\xi)-\lambda) \cF u = \cF f
$$
and $\Re \,\Phi(-\xi)\geq 0$, we have
\begin{eqnarray*}
 \cF u (\xi)  &=&  - \frac{1}{\Phi(-\xi)+\lambda }\cF f (\xi) \\
&=&  - \Big( \int_0^\infty    e^{-t\Phi(-\xi) -\lambda t} ~dt ~ \cF f(\xi) \Big)\\
&=&  - \Big( \int_0^\infty    \int_{\bR^d} e^{- i \xi \cdot x} p_\Phi(t,dx) e^{-\lambda t} ~dt ~ \cF f(\xi) \Big)\\
&=&  -\cF\Big(\int_0^\infty (p_\Phi(t , \cdot) \ast f (x) ) e^{-\lambda t}~dt\Big)(\xi).
\end{eqnarray*}
Therefore,
$$
u(x) = -\int_0^\infty (p_\Phi(t , \cdot) \ast f ) e^{-\lambda t}~dt
$$
and by Young's inequality,
\begin{eqnarray*}
\|u\|_{L_p} \leq   \int_0^\infty \int_{\bR^d} p_{\Phi}(t,dx) e^{-\lambda t}~dt \|f\|_{L_p} \leq \frac{1}{\lambda} \|f\|_{L_p}.
\end{eqnarray*}
Hence  the lemma is proved.
\end{proof}

\begin{definition}
We write $u \in \cH_p^\cA$  if and only if there exists a sequence of functions $u_n \in C_0^\infty$ such that   $u_n \to u$ in $L_p$ and $\{\cA u_n:n=1,2,\cdots\}$ is a cauchy sequence in $L_p$.
By $\cA u$ we denote the limit of $\cA u_n$ in $L_p$.
\end{definition}
\begin{lemma}
$\cH_p^\cA$ is a Banach space equipped with the norm
\begin{eqnarray*}
\|u\|_{\cH_p^\cA} :=   \|u\|_{L_p} + \|\cA u\|_{L_p}.
\end{eqnarray*}
\end{lemma}
\begin{proof}
It is obvious.
\end{proof}

\begin{definition}
We say that $u\in \cH_p^\cA$ is a solution of the equation
\begin{eqnarray}\label{maineqn1}
Lu - \lambda u =f \quad\quad  ~\text{in}\,\,~\bR^d
\end{eqnarray}
 if and only if there exists a sequence $\{u_n \in C_0^\infty\}$ such that $u_n$ converges to $u$ in $\cH_p^\cA$ and $Lu_n  -  \lambda u_n$ converges to $f$ in $L_p$.
Similarly, we consider the equation
\begin{eqnarray}        \label{maineqn2}
\tilde Lu - \lambda u =f \quad\quad  ~\text{in}\,\,~\bR^d
\end{eqnarray}
in the same sense.
\end{definition}

\begin{lemma}[Maximum principle] \label{max}
Let $\lambda >0$, $b(x)$ be an $\bR^d$-valued bounded function on $\bR^d$ and $u$ be a function in $C^2_b$ satisfying $u(x) \to 0$ as $|x| \to \infty$. If $Lu+b(x)\cdot \nabla u-\lambda u =0$ in $\bR^d$, then $u \equiv 0$. Also, the same statement is true with $\tilde{L}$ in place of $L$.
\end{lemma}

\begin{proof}
Suppose that $u$ is not identically zero. Without loss of generality, assume $\sup_{\bR^d} u >0$ (otherwise consider $-u$). Since $u$ goes to zero as $|x| \to \infty$, there exists  $x_0 \in \bR^d$ such that $u(x_0) = \sup_{\bR^d} u$. Thus $\nabla u (x_0)=0$ and
\begin{eqnarray*}
Lu(x_0)=\int_{\R^d}\left(u(x_0+y)-u(x_0)- y \cdot \nabla  u(x_0)  \chi (y) \right) a(y)J(y)\, dy\leq 0.
\end{eqnarray*}
Therefore we reach the contradiction. Indeed,
\begin{eqnarray*}
Lu (x_0) +b(x_0)\cdot \nabla u(x_0)-\lambda u(x_0) < 0.
\end{eqnarray*}
The proof for $\tilde{L}$ is almost identical. The lemma is proved.
\end{proof}

This maximum principle yields the denseness of $(L+b \cdot \nabla -\lambda) C_0^\infty$ and  $(\tilde{L}+b \cdot \nabla -\lambda) C_0^\infty$ in $L_p$.

\begin{lemma}\label{exis}
Let $\lambda >0$ and $b \in \bR^d$ be independent of $x$. Then $(L+b\cdot \nabla -\lambda) C_0^\infty:= \{ Lu +b \cdot \nabla u - \lambda u : u \in C_0^\infty\}$ is dense in $L_p$ for any $p \in (1,\infty)$. Also, the same statement holds with $\tilde{L}$ in place of $L$.
\end{lemma}

\begin{proof}
Due to the similarity we only prove the first statement.   Suppose that the statement  is false. Then by the  Hahn-Banach theorem and Riesz's representation theorem,
there exists a nonzero $v \in L_{p/(p-1)}$ such that
\begin{eqnarray}
\label{4081}\int_{\bR^d} \left( Lu(x) + b \cdot \nabla u (x)- \lambda u(x) \right) v(x)~dx =0
\end{eqnarray}
for all $u \in C_0^\infty$.

Fixing $y \in \bR^d$, we apply \eqref{4081} with $u( y- \cdot)$. Then, due to Fubini's Theorem,
\begin{eqnarray*}
0
&=&\int_{\bR^d} \left( L^\ast u(y-x) - b \cdot \nabla u (y-x)- \lambda u(y-x) \right) v(x)~dx \\
&=&L^\ast u \ast v(y) - b \cdot ( \nabla u \ast v (y))  - \lambda u \ast v (y) =(L^\ast-b \cdot \nabla -\lambda)(u \ast v)(y).
\end{eqnarray*}
Therefore from the previous lemma, we have $u \ast v  = 0$ for any $u\in C^{\infty}_0$.  Therefore,  $v = 0$ $(a.e.)$ and we have a contradiction.
\end{proof}


\begin{corollary}[Uniqueness] \label{unique}
Let $\lambda >0$. Suppose that there exist $u, v \in \cH^\cA_p$ satisfying
$$
Lu  -\lambda u = 0, \quad \tilde{L}v  -\lambda v = 0.
$$
Then $u =v= 0$.
\end{corollary}
\begin{proof}
 By the definition of a solution and the assumption of this corollary, there exists a sequence $\{ u_n  \in C_0^\infty\}$ such that for all $w \in C_0^\infty$
\begin{eqnarray*}
0= \int_{\bR^d} \lim_{n \to \infty} (L u_n  -\lambda u_n) w ~dx = \int_{\bR^d} u (L^\ast w  - \lambda w)~dx.
\end{eqnarray*}
Since $\{ L^\ast w - \lambda w : w \in C_0^\infty\}$  is dense in $L_{p/(p-1)}$ owing to Lemma \ref{exis}, we conclude
$u=0$, and by the same argument  we have $v=0$.
\end{proof}

Here is our $L_2$-theory. We emphasize that  only  (\ref{elliptic con}) and (\ref{levy measure}) are assumed for the $L_2$-theory.
The proof of Theorem \ref{main L_2} is given in Section \ref{s:2}.

\begin{theorem}
                              \label{main L_2}
 Let $\lambda>0$. Then for any $f \in L_2$ there exist unique solutions
$u,v \in \cH_2^\cA$ of equation \eqref{maineqn1} and \eqref{maineqn2} respectively, and for these solutions we have
\begin{eqnarray}
               \label{L21}
\| \cA u \|_{L_2} + \lambda \| u \|_{L_2} \leq N (d, \nu , \Lambda) \|f\|_{L_2},
\end{eqnarray}
\begin{eqnarray}
                           \label{L22}
\| \cA v \|_{L_2} + \lambda \| v \|_{L_2} \leq N (d, \nu , \Lambda) \|f\|_{L_2}.
\end{eqnarray}
\end{theorem}
  The issue regarding the continuity of $L$ (or $\tilde{L}$) : $\cH^{\cA}_p \to L_p$ will be discussed later.


\vspace{3mm}
 For the case $p\neq 2$, we consider the following    conditions on $J(y)=j(|y|)$ :
\vspace{1mm}

({\bf H1}): There exist constants $\kappa_1>0$ and $\alpha_0 >0$ such that
\begin{eqnarray}
                              \label{ma}
j(t) \leq \kappa_1 (s/t)^{d+\alpha_0} j(s), \quad \forall \,\,0<s\leq t.
\end{eqnarray}
Moreover, $\alpha_0 \leq 1 $ if $\sigma \leq 1$ and  $1<\alpha_0 < 2$ if $\sigma > 1$.

\vspace{1mm}

({\bf H2}): There exists a constant $\kappa_2>0$  such that for all $ t > 0$,

\begin{eqnarray}
                 \label{ma2}
\int_{|y|\leq 1} |y| j(t |y|) ~dy \leq \kappa_2 j(t) \quad \quad \text{if}\,\, \sigma\in (0,1),
\end{eqnarray}
\begin{eqnarray}
                   \label{ma3}
\int_{|y|\leq 1} |y|^{2 } j(t |y|) ~dz \leq \kappa_2 j(t) \quad \quad \text{if}\,\, \sigma\geq 1.
\end{eqnarray}


\begin{remark}      \label{h1remark}
(i) By taking $t=1$ in (\ref{ma}),
\begin{equation}
                  \label{eqn 5.22.3}
j(1)\kappa^{-1}_1 s^{-d-\alpha_0}\leq j(s), \quad \forall\,\, s\in (0,1).
\end{equation}
An upper bound of $j(s)$ near $s=0$ is obtained in the following lemma.

(ii) {\bf{H1}} and {\bf{H2}} are needed even to guarantee the continuity of the operator $L: \cH^{A}_2 \to L_2$ (see Lemma \ref{l2}).


\end{remark}

\begin{lemma}\label{aslemma}

Suppose
\begin{eqnarray}   \label{601-1}
j(s)\geq C j(t), \quad \forall\, s\leq  t,
\end{eqnarray}
and {\bf H2} hold. Then
there exists a constant $N(d,\kappa_2,C)>0$  such that for all $0<s\leq t$
\begin{eqnarray}
                  \label{eqn to}
j(t) \geq N (s/t)^{d+1} j(s) \quad  (\text{if}\,\,\sigma < 1), \quad
j(t) \geq N (s/t)^{d+2} j(s) \quad  (\text{if}\,\,\sigma \geq 1).
\end{eqnarray}
 On the other hand, if there exists $\alpha>0$ so that   $\alpha<1$ if $\sigma<1$,  $\alpha<2$ if $\sigma\geq 1$, and
\begin{eqnarray}
                                              \label{as2}
j(t) \geq N (s/t)^{d+\alpha} j(s), \quad \forall\, 0<s\leq t,
\end{eqnarray}
then {\bf H2} holds.
\end{lemma}


\begin{remark}
By Lemma \ref{aslemma}, both {\bf{H1}} and {\bf{H2}} hold if $0<\alpha_0\leq \alpha$ and
$$
N^{-1} (s/t)^{d+\alpha} j(s)\leq  j(t) \leq N (s/t)^{d+\alpha_0} j(s), \quad \forall \,\,\,0<s\leq t.
$$

\end{remark}

\begin{example}
                 \label{exam 1}
  Let  $J(y)=j(|y|)$ be defined as in (\ref{jum}), that is for a Bernstein function $\phi(\lambda )=\int_{\bR} (1-e^{-\lambda t})\mu(dt)$ and $u\in C^2_0
  $,
$$
 j(r) = \int_0^\infty (4\pi t)^{-d/2} e^{-r^2/(4t)}~\mu(dt),
 $$
 and
\begin{eqnarray*}
\phi(\Delta)u &=& \int_{\bR^d} \left(u(x+y) -u(x) -\nabla u (x) \cdot y I_{|y|\leq 1} \right)J(y)~dy\\
&=&-\cF(\phi(|\xi|^2)\cF(u)(\xi)).
\end{eqnarray*}
Then, {\bf{H1}} and {\bf{H2}} are satisfied if
  $\phi$ is given, for instance, by any one of
\begin{itemize}
\item[(1)] $\phi(\lambda)=\sum_{i=1}^n\lambda^{\alpha_i}$, $0<\alpha_i<1$;
\item[(2)] $\phi(\lambda)=(\lambda+\lambda^\alpha)^\beta$, $\alpha, \beta\in (0, 1)$;
\item[(3)] $\phi(\lambda)=\lambda^\alpha(\log(1+\lambda))^\beta$, $\alpha\in (0, 1)$,
$\beta\in (0, 1-\alpha)$;
\item[(4)] $\phi(\lambda)=\lambda^\alpha(\log(1+\lambda))^{-\beta}$, $\alpha\in (0, 1)$,
$\beta\in (0, \alpha)$;
\item[(5)] $\phi(\lambda)=(\log(\cosh(\sqrt{\lambda})))^\alpha$, $\alpha\in (0, 1)$;
\item[(6)] $\phi(\lambda)=(\log(\sinh(\sqrt{\lambda}))-\log\sqrt{\lambda})^\alpha$, $\alpha\in (0, 1)$.
\end{itemize}
This is because all these functions satisfy the conditions
\begin{itemize}
\item[{\bf{A}}:] \, $\exists \, 0 < \delta_1 \leq \delta_2 <1$,
$$
N^{-1}\lambda^{\delta_1} \phi (t) \leq  \phi (\lambda t) \leq N \lambda^{\delta_2} \phi(t), \quad \forall\,\lambda \geq 1, t \geq 1
$$
\item[{\bf{B}}:] \, $\exists\, 0 < \delta_3 \leq \delta_4 < 1$,
$$
N^{-1} \lambda^{\delta_3} \phi (t) \leq \phi (\lambda t) \leq N \lambda^{\delta_4} \phi(t), \quad \forall\, \lambda \leq 1, t \leq 1,
$$
\end{itemize}
and under these condition one can prove (see \cite{KSV8})
\begin{eqnarray*}
N^{-1} \Big( \frac{R}{r} \Big)^{\delta_1 \wedge \delta_3} \leq  \frac{\phi(R)}{\phi(r)}  \leq N \Big( \frac{R}{r} \Big)^{\delta_2 \vee \delta_4}
\end{eqnarray*}
and
\begin{eqnarray}
                            \label{jumpestimate}
N^{-1}\phi(|y|^{-2}) |y|^{-d} \leq J(y) \leq N\phi(|y|^{-2}) |y|^{-d},
\end{eqnarray}
and consequently our conditions  {\bf{H1}} and {\bf{H2}} hold. One can easily construct concrete examples of $j(r)$ using (\ref{jumpestimate}) and $(1)$-$(6)$ (just replace $\lambda$ by $r^{-2}$).  See  the tables at the end of \cite{SSV} for more examples satisfying {\bf{A}} and {\bf{B}}.

\end{example}


\begin{remark}
If $p\neq 2$, our $L_p$-theory  does not cover the case when the jump function $J(y)$ is related to the relativistic $\alpha$-stable process with mass $m>0$ (i.e. a subordinate Brownian motion  with the infinitesimal generator $\phi(\Delta)=m-(m^{2/{\alpha}}-\Delta)^{\alpha/2}$). This is because
the related jump function decreases exponentially  fast at the infinity (for instance, see \cite{chen}) and thus condition {\bf{H2}} fails (see (\ref{eqn to})).
\end{remark}

\vspace{5mm}

{\bf{Proof of Lemma \ref{aslemma}}}.
Assume \eqref{601-1} and {\bf H2} hold. We put $B_1=\cup_{n=0}^\infty B_{(n)}$, where $B_{(n)} = B_{2^{-n}} \setminus B_{2^{-(n+1)}}$.
Due to \eqref{601-1} for each $n\geq 0$,
\begin{eqnarray*}
\kappa_2 j(t)
\notag &\geq& \int_{|y| \leq 1} |y|^2 j( t |y|) ~dy
= \sum_{n=0}^\infty \int_{B(n)} |y|^2 j(t|y|)~dy   \\
&\geq& N\sum_{n=0}^\infty2^{-(n+1)(d+2)}j( t2^{-n})
\geq N2^{-(n+1)(d+2)}j( t2^{-n}).
\end{eqnarray*}
Put $s=t\lambda $, where  $\lambda \in (0,1)$, and take an
integer $m(\lambda)\geq 0$ such that $2^{-(m+1)} \leq \lambda \leq 2^{-m}$.
Then by \eqref{601-1},
\begin{eqnarray*}
 j( t)
\geq N 2^{-(m+2)(d+2)} j(2^{-(m+1)} t)
\geq N \lambda^{d+2} j(\lambda t).
\end{eqnarray*}
Similarly, $j(\lambda t) \leq \lambda^{-d-1} j(t)$ if $\sigma <1$.

For the other direction, put $s = t|y|$ in \eqref{as2}.  If $\sigma < 1$ then
\begin{eqnarray*}
\int_{|y| \leq 1}  |y| j(t |y|) ~dy
&\leq& Nj(t)\int_{|y| \leq 1}  |y| \frac{j(t |y|)}{j(t)} ~dy\\
&\leq& N j(t) \int_{|y| \leq 1}  |y|^{-d-\alpha_1+1} ~dy \leq Nj(t)
\end{eqnarray*}
 and otherwise, that is, if $\sigma \geq 1$ then
\begin{eqnarray*}
\int_{|y| \leq 1}  |y|^2 j(t |y|) ~dy
&\leq& Nj(t)\int_{|y| \leq 1}  |y|^2 \frac{j(t |y|)}{j(t)} ~dy\\
&\leq& N j(t) \int_{|y| \leq 1}  |y|^{-d-\alpha_2+2} ~dy \leq Nj(t).
\end{eqnarray*}
The lemma is proved. \qed

\vspace{5mm}

Define
$$
\Psi (\xi) := -\int_{\bR^d} (e^{i \xi  \cdot y} -1 - i (y \cdot \xi )\chi(y)) J(y) dy= \int_{\bR^d} (1- \cos \xi \cdot y)  J(y) dy.
$$
Then
$$
\cA u= \cF^{-1} ( - \Psi(\xi) \cF u),\quad \quad \forall\, u\in C^{\infty}_0.
$$
By abusing the notation, we also use $\Psi(|\xi|)$ instead of $\Psi(\xi)$ because $\Psi(\xi)$ is rotationally invariant.

The following result will be used to prove the continuity of the operator $L$.
\begin{lemma} \label{h3}
Suppose that \eqref{601-1} holds. Then
there exists a constant $N(d,C)>0$ such that for all $\xi \in \bR^d$
\begin{eqnarray} \label{4251}
j(|\xi|) \leq N |\xi|^{-d}  \Psi(|\xi|^{-1}).
\end{eqnarray}
\end{lemma}
\begin{proof}
By  \eqref{601-1},
\begin{eqnarray*}
\Psi ( |\xi|^{-1})
&=& \int_{\bR^d} (1- \cos( y^1/|\xi|))  J(y) ~dy
= |\xi|^d\int_{\bR^d} (1- \cos( y^1))  J(|\xi|y) ~dy\\
&\geq& |\xi|^d\int_{|y| \leq 1} (1- \cos( y^1))  J(|\xi|y) dy \\
&\geq& C j(|\xi|)|\xi|^d\int_{|y|\leq 1} (1- \cos( y^1))   ~dy
\geq  N j(|\xi|)|\xi|^d.
\end{eqnarray*}
Hence the lemma is proved.
\end{proof}

The following condition will be considered for the case $\sigma=1$.  This condition is needed even  to prove the continuity of $L$.

\begin{assumption}
                               \label{cancallation}
If $\sigma =1$ then
\begin{eqnarray}\label{cancel}
\int_{\partial B_r} y^i a(y) J(y) dS_r(y) =0, \quad \forall r \in (0,\infty),\, i=1,\cdots, d,
\end{eqnarray}
where $dS_r$ is the surface measure on $\partial B_r$.
\end{assumption}


Here is our $L_p$-theory for equation (\ref{eqn Lp1}) below.

\begin{theorem}
\label{maint}
Suppose that  {\bf H1} and {\bf H2} hold and  Assumption \ref{cancallation} also holds if $\sigma=1$. Let $\lambda > 0$ and $p>1$. Then for any $f \in L_p$ there exists a unique solution
$u \in \cH_p^\cA$ of the equation
\begin{equation}
                      \label{eqn Lp1}
Lu-\lambda u =f,
\end{equation}
 and for this solution we have
\begin{eqnarray}
\label{mainin}\| \cA u \|_{L_p} + \lambda \| u \|_{L_p} \leq N (d, p,\nu , \Lambda, J) \|f\|_{L_p}.
\end{eqnarray}
Moreover,  $L$ is a continuous operator from $\cH_p^\cA$ to $L_p$, and
\eqref{mainin} holds for all $u \in \cH_p^\cA$ with $f:=Lu - \lambda u$.
\end{theorem}
The proof of this theorem will be given in Section \ref{s:5}.

\begin{remark}
Since the constant $N$ in \eqref{mainin} does not depend on $\lambda$, for any $u \in \cH_p^\cA$
\begin{eqnarray*}
\| \cA u \|_{L_p}  \leq N \|Lu\|_{L_p}.
\end{eqnarray*}
\end{remark}


 To study the equations  with the operator $\tilde{L}$, we consider an additional condition, which always holds when $\sigma=1$.
\vspace{2mm}

\begin{assumption}[{\bf H3}]
            \label{assump 10.25.3}
 Any one of the following (i)-(iv) holds:
\vspace{3mm}

(i)  $\cA$ is a higher order differential operator than $I_{\sigma \neq 1} \nabla u$, that is for any $\varepsilon>0$ there exists $N(\varepsilon)>0$ so that for any $u\in C^{\infty}_0$
\begin{equation}
                           \label{eqn 5.20.1}
   I_{\sigma \neq 1} \|\nabla u\|_p\leq \varepsilon \|\cA u\|_p+N(\varepsilon)\|u\|_p.
 \end{equation}

 (ii) $\sigma <1$ and
\begin{equation}
                                    \label{eqn 5.16.1}
\int_{r\leq |y|\leq 1}y^i\Big(a(y)-[a(y)\wedge a(-y)]\Big)\,J(y)dy=0, \quad \forall  \,r\in (0,1),\, i=1,\cdots,d.
\end{equation}

(iii) $\sigma < 1$ and  there exists a constant $\kappa_3>0$  such that for all $ 0<t<1$,
\begin{eqnarray}
                 \label{ma4}
\int_{|z| \geq 1} |z| j(t |z|) ~dz \leq \kappa_2 j(t).
\end{eqnarray}
\end{assumption}

(iv) $\sigma >1$ and
\begin{equation}
                                    \label{eqn 5.25.1}
\int_{1\leq |y|\leq r}y^i\Big(a(y)-[a(y)\wedge a(-y)]\Big)\,J(y)dy=0, \quad \forall  \,r>1\,\, i=1,\cdots,d.
\end{equation}

\vspace{3mm}

\begin{remark}

(i) Note that (\ref{eqn 5.20.1}) is satisfied if for some $\alpha>1$,
\begin{equation}
                  \label{eqn 5.23.1}
\|\Delta^{\alpha/2}u\|_p\leq N(\|u\|_p+\|\cA u\|_p), \quad \forall u\in C^{\infty}_0,
\end{equation}
or, equivalently $|\xi|^{\alpha}(1+\Psi(\xi))^{-1}$ is a $L_p$-Fourier multiplier. Thus, certain differentiability of $J(y)$ is required (see Lemma \ref{lemma 5.25.5} below).

(ii) It is easy to check that (\ref{ma4}) holds if for a $\alpha >1 $,
\begin{eqnarray}            \label{h3r}
 j(\lambda t) \leq N \lambda^{-d-\alpha} j(t), \quad \forall \,\,\lambda \in (1,\infty), \,\, 0<t<1.
\end{eqnarray}

(iii) Obviously, (\ref{eqn 5.16.1}) holds if $a(y)=a(-y)$ for $|y|\leq 1$, and (\ref{eqn 5.25.1}) holds if $a(y)=a(-y)$ for $|y|\geq 1$.
\end{remark}


Below we give  a  sufficient condition for (\ref{eqn 5.20.1}).

\begin{lemma}
                 \label{lemma 5.25.5}
(i) {\bf{H3}}-(i) holds if  $\cA=\phi(\Delta)$ for some Bernstein function $\phi$ satisfying
 \begin{equation}
               \label{eqn 5.22.1}
 1+\phi(|\xi|^2)\geq N|\xi|^{\alpha}, \quad \forall \xi\in \bR^d,
 \end{equation}
 where $\alpha>1$ and $N>0$.

 (ii) All of {\bf{H1, H2}} and {\bf{H3}} hold if $\sigma>1$, $\cA=\phi(\Delta)$ and $\phi$ satisfies conditions {\bf{A}} and {\bf{B}} described in Example \ref{exam 1}.


\end{lemma}

\begin{proof}
(i). Let $\phi(\lambda)=\int_{\bR}(1-e^{-\lambda t})\mu(dt)$, where $\int_{\bR}(1\wedge |t|) \mu(dt)<\infty$. Then from $t^n e^{-t}\leq N(n) (1-e^{-t})$, we get
\begin{equation}
                        \label{eqn 5.22}
|\lambda|^n |D^n \phi(\lambda)|\leq N \phi(\lambda).
\end{equation}
 For any $u\in C^{\infty}_0$,
$$
\cA u =\cF^{-1}(\phi(|\xi|^2) \cF(u)(\xi)),
$$
$$
 \Delta^{\alpha/2}u=\cF^{-1}( |\xi|^{\alpha} \cF(u)(\xi))
=\cF^{-1}( \eta(\xi) (1+\phi(|\xi|^2) \cF(u)(\xi)),
$$
where $\eta(\xi)=|\xi|^{\alpha}(1+\phi(|\xi|^2))^{-1}$. Using  (\ref{eqn 5.22.1}) and (\ref{eqn 5.22}), one can easily check
$$
|D^n \eta(\xi)|\leq N(n)|\xi|^{-n}, \quad \forall\, \xi,
$$
and therefore $\eta$ is a Fourier multiplier (see Theorem IV.3.2 of \cite{Ste2}) and
$$
\|\Delta^{\alpha/2}u\|\leq N(\|u\|_p+\|\cA u\|_p),
$$
$$
\|\nabla u\|_p\leq \varepsilon  \|\Delta^{\alpha/2}u\|_p+N(\varepsilon)\|u\|_p\leq N \varepsilon \|\cA u\|_p+N\|u\|_p.
$$

(ii) If  {\bf{A}} and {\bf{B}} hold, then as explained before both {\bf{H1, H2}} hold, and we also have (see (\ref{jumpestimate})),
$$
N^{-1}\phi(|y|^{-2}) |y|^{-d} \leq J(y) \leq N\phi(|y|^{-2}) |y|^{-d}.
$$
Thus if $|\xi|\geq 1$, then
$$
\phi(|\xi|^2)\geq N |\xi|^{-d} J(|\xi|^{-1})\geq N|\xi|^{\alpha_0},
$$
where (\ref{eqn 5.22.3}) is used for the last inequality. Hence the lemma is proved.
\end{proof}

Here is our $L_p$-theory for equation (\ref{eqn Lp2}) below.

\begin{theorem}
                                            \label{maint2}
Suppose that {\bf H1}, {\bf H2} and {\bf H3} hold  and Assumption \ref{cancallation} also holds if $\sigma=1$. Let $\lambda > 0$ and $p>1$. Then for any $f \in L_p$ there exists a unique solution
$u \in \cH_p^\cA$ of the equation
\begin{equation}
                       \label{eqn Lp2}
\tilde{L}u-\lambda u=f,
\end{equation}
 and for this solution we have
\begin{eqnarray}    \label{mainieq2}
\| \cA u \|_{L_p} + \lambda \| u \|_{L_p} \leq N (d, \nu , \Lambda, \lambda, J) \|f\|_{L_p}.
\end{eqnarray}
\end{theorem}
The proof of this theorem will be given in  Section \ref{s:5}. Actually the constant $N$ in (\ref{mainieq2}) is independent of $\lambda$ except the case when {\bf{H3}}(i) is assumed.



%

\section{$L_2$-theory}
                         \label{s:2}
In this section we prove \eqref{L21} and \eqref{L22}. These estimates and Lemma \ref{exis} yield the  unique solvability of equations \eqref{maineqn1} and \eqref{maineqn2}.
The Fourier transform and Parseval's identity are used to prove these estimates.

\begin{lemma}
                                     \label{l2}
Let $\lambda \geq 0$ be a constant.

(i) For any $u \in C_0^\infty$
\begin{eqnarray}
\label{4083}\| \cA u \|_{L_2} + \lambda \| u \|_{L_2} \leq N (d, \nu) \|Lu - \lambda u \|_{L_2}
\end{eqnarray}
and
\begin{eqnarray}
                                      \label{40832}
\| \cA u \|_{L_2} + \lambda \| u \|_{L_2} \leq N (d, \nu) \|\tilde Lu - \lambda u \|_{L_2}.
\end{eqnarray}

(ii) Let {\bf{H1}} hold and $\sigma > 1$. Then both $L$ and $\tilde{L}$ are continuous operators from $\cH_2^\cA$ to $L_2$, and for any
$u\in C^{\infty}_0$,
\begin{equation}
                                \label{4093}
\|Lu\|_{L_2}\leq N\|\cA u\|_{L_2}, \quad \quad \|\tilde{L}u\|_{L_2}\leq N\|u\|_{\cH_2^{\cA}},
\end{equation}
where $N=N(d,\nu,J)$. Moreover, \eqref{4083} and \eqref{40832} hold for any $u \in \cH_2^\cA$.

(iii) Let  {\bf{H1}} and {\bf{H2}} hold, and   Assumption \ref{cancallation} also hold if $\sigma=1$. Then the claims of (ii) hold for $L$ (not for $\tilde{L}$) for any $\sigma\in (0,1]$.
\end{lemma}

\begin{proof}
(i). Let $u\in C^{\infty}_0$. Taking the Fourier transform, we get
\begin{eqnarray}
                                 \label{4094}
\cF (Lu) (\xi) = \cF u (\xi) \int_{\bR^d} (e^{i \xi \cdot y} -1 - i y \cdot \xi \chi(y)) a(y)J(y) dy.
\end{eqnarray}
By Parseval's identity,
\begin{eqnarray*}
&&\int_{\bR^d} |L u(x)|^2 dx
= (2\pi)^{-d}\int_{\bR^d} |\cF(L u) (\xi)|^2 d\xi \\
&\geq& (2\pi)^{-d} \int_{\bR^d} | \cF u (\xi)|^2 \left|\Re  \int_{\bR^d} (e^{i \xi \cdot y} -1 - i y \cdot \xi \chi(y)) a(y)J(y) dy\right|^2 d\xi \\
&=& (2\pi)^{-d} \int_{\bR^d} | \cF u (\xi)|^2 \left| \int_{\bR^d} (1 - \cos(\xi \cdot y )) a(y)J(y) dy\right|^2 d\xi \\
&\geq& (2\pi)^{-d} \nu^2 \int_{\bR^d} | \cF u (\xi)|^2 \left| \int_{\bR^d} (1 - \cos(\xi \cdot y )) J(y) dy\right|^2 d\xi \\
&=& \nu^2 \int_{\bR^d} | \cA u |^2 dx,
\end{eqnarray*}
where the facts that $1-\cos(\xi \cdot y)$ is nonnegative and $a(y) \geq \nu$ are used above.

Similarly, since $uLu$ is real,
\begin{eqnarray*}
&&-\int_{\bR^d} u L u ~dx
= - (2\pi)^{-d} \int_{\bR^d} \cF(Lu)(\xi) \overline{ \cF (u) (\xi)}~d\xi\\
&=& - (2\pi)^{-d} \int_{\bR^d}| \cF(u)(\xi)|^2  \Re \int_{\bR^d}\left(e^{i \xi \cdot y} - 1 - i y \cdot \xi \chi^{(\sigma)} (y) \right) a(y) J(y) ~dyd\xi\\
&=&  (2\pi)^{-d} \int_{\bR^d}| \cF(u)(\xi)|^2   \int_{\bR^d}\left(1-\cos (\xi \cdot y) \right) a(y) J(y) ~dyd\xi\\
&\geq&  \frac{\nu}{2} (2\pi)^{-d} \int_{\bR^d}| \cF(u)(\xi)|^2   \int_{\bR^d}\left(1-\cos (\xi \cdot y) \right)  J(y) ~dyd\xi\\
&=&  -\frac{\nu}{2} \int_{\bR^d} u \cA u ~dx.
\end{eqnarray*}
 Hence,
\begin{eqnarray*}
&&\int_{\bR^d} |Lu - \lambda u |^2~dx \\
&=& \int_{\bR^d} |Lu|^2~dx-2 \lambda \int_{\bR^d} uLu~dx+\lambda^2 \int_{\bR^d} |u|^2~dx \\
&\geq& \nu^2 \int_{\bR^d} |\cA u|^2~dx-\lambda \nu\int_{\bR^d} u \cA u~dx+\lambda^2 \int_{\bR^d} |u|^2~dx \\
&\geq& \nu^2 \int_{\bR^d} |\cA u|^2~dx-\frac{\nu^2}{2}\int_{\bR^d} u^2 ~dx-\frac{\lambda^2}{2}\int_{\bR^d}  |\cA u|^2 ~dx+\lambda^2 \int_{\bR^d} |u|^2~dx \\
&=& \frac{\nu^2}{2}\int_{\bR^d} |\cA u|^2~dx + \frac{\lambda^2}{2} \int_{\bR^d} |u|^2~dx.
\end{eqnarray*}
Thus \eqref{4083} holds. Also, \eqref{40832} is proved similarly.

(ii)-(iii). Next, we prove   (\ref{4093}) for  any $u\in C^{\infty}_0$.
Unlike the case $j(r)=r^{-d-\alpha}$, the proof is not completely trivial. Condition {\bf{H1}} is needed if $\sigma> 1$, and {\bf{H2}} is additionally needed if $\sigma\leq 1$.

By using \eqref{4094} and Parseval's identity again,
\begin{eqnarray*}
&&\int_{\bR^d} |L u(x)|^2 dx
= (2\pi)^{-d} \int_{\bR^d} |\cF(L u) (\xi)|^2 d\xi \\
&=& (2\pi)^{-d}\Big[\int_{\bR^d} | \cF u (\xi)|^2 \left|\Re  \int_{\bR^d} (e^{i \xi \cdot y} -1 - i y \cdot \xi \chi(y)) a(y)J(y) ~dy\right|^2 d\xi \\
&&+\int_{\bR^d} | \cF u (\xi)|^2 \left|\Im  \int_{\bR^d} (e^{i \xi \cdot y} -1 - i y \cdot \xi \chi(y)) a(y)J(y) ~dy\right|^2 d\xi \Big]  \\
&\leq& (2\pi)^{-d}\int_{\bR^d} | \cF u (\xi)|^2 \left| \int_{\bR^d} (1 - \cos(\xi \cdot y )) a(y)J(y) ~dy\right|^2 d\xi \\
&&+(2\pi)^{-d}\int_{\bR^d} | \cF u (\xi)|^2 \left| \int_{|y| |\xi| \geq 1} (\sin (\xi \cdot y) - y \cdot \xi \chi(y)) a(y)J(y) ~dy\right|^2 d\xi \\
&&+(2\pi)^{-d}\int_{\bR^d} | \cF u (\xi)|^2 \left| \int_{|y||\xi| < 1} (\sin (\xi \cdot y) - y \cdot \xi \chi(y)) a(y)J(y) ~dy\right|^2 d\xi \\
&:=&\cI_1+\cI_2+\cI_3.
\end{eqnarray*}
Similarly,
$$
\int_{\bR^d}|\tilde{L}u|^2 dx=\tilde{\cI}_1+ \tilde{\cI}_2+\tilde{\cI}_3,
$$
where $\tilde{\cI}_i$ are obtained by replacing $\chi(y)$ in $\cI_i$ with $I_{B_1}(y)$.
Here $\cI_1$ and $\tilde{\cI}_1$ are easily controlled by $N \|\cA u\|_{L_2}^2$.

Due to {\bf H1}, \eqref{cancel}, the definition of $\chi$, and the change of variables $y \to \frac{y}{|\xi|}$,
\begin{eqnarray*}
\cI_2
&\leq& N \int_{\bR^d} | \cF u (\xi)|^2 |\xi|^{-2d}
\left| \int_{|y| \geq 1} (\sin ( \frac{\xi}{|\xi|} \cdot y) - y \cdot \frac{\xi}{|\xi|} \chi(\frac{y}{|\xi|})) a(\frac{y}{|\xi|})J( \frac{y}{|\xi|}) ~dy\right|^2 d\xi \\
&\leq& N \int_{\bR^d} | \cF u (\xi)|^2 |\xi|^{-2d}j(1/|\xi|)^2 \\
&&\quad \quad \quad\quad \times \left( \int_{|y| \geq 1} \left|\sin ( \frac{\xi}{|\xi|} \cdot y) -  I_{\sigma \neq 1}y \cdot \frac{\xi}{|\xi|} \chi(  \frac{y}{|\xi|})\right| a( \frac{y}{|\xi|})|y|^{-d-\alpha_0} ~dy\right)^2 d\xi \\
&\leq& N \int_{\bR^d} | \cF u (\xi)|^2 |\xi|^{-2d}j(1/|\xi|)^2~d \xi.
\end{eqnarray*}
Hence, by Lemma \ref{h3},
\begin{eqnarray*}
\cI_{2}
&\leq&  N \int_{\bR^d} | \cF u (\xi)|^2 (\Psi(\xi))^2~d  \xi =  N \int_{\bR^d} |\cA u|^2~dx.
\end{eqnarray*}
Similarly, if $\sigma > 1$,
\begin{eqnarray*}
\tilde{\cI}_2 &\leq& N \int_{\bR^d} | \cF u (\xi)|^2 |\xi|^{-2d}j(1/|\xi|)^2 \\
&&\quad \quad   \times \left( \int_{|y| \geq 1} \left|\sin ( \frac{\xi}{|\xi|} \cdot y) -  I_{\sigma > 1}y \cdot \frac{\xi}{|\xi|} I_{|y|\leq|\xi|}\right| \,\,a( \frac{y}{|\xi|})|y|^{-d-\alpha_0} ~dy\right)^2 d\xi\\
&\leq& N \int_{\bR^d} | \cF u (\xi)|^2 |\xi|^{-2d}j(1/|\xi|)^2~d \xi
\leq N\int_{\bR^d} |\cA u|^2~dx.
\end{eqnarray*}

Also, using  the fundamental theorem of calculus, the definition of $\chi$ and \eqref{cancel},
\begin{eqnarray*}
\cI_3
&\leq& N \int_{\bR^d} | \cF u (\xi)|^2 \left| \int_{ |y| |\xi|< 1 } (\sin (\xi \cdot y) - y \cdot \xi \chi(y)) a(y)J(y) ~dy\right|^2 d\xi \\
&=& N \int_{\bR^d} | \cF u (\xi)|^2
 \left| \int_{|y| |\xi| <1 } \int_0^1 \frac{d}{dt}(\sin (t \xi \cdot y) - t y \cdot \xi \chi(y))~dt~ a(y)J(y) ~dy\right|^2 d\xi \\
&=& N \int_{\bR^d} | \cF u (\xi)|^2
\left| \int_{|y| |\xi| < 1} (\xi \cdot y) \int_0^1 (\cos (t \xi \cdot y) -  \chi(y))~dt~ a(y)J(y) ~dy\right|^2 d\xi \\
&=& I_{\sigma \leq 1}N \int_{\bR^d} | \cF u (\xi)|^2
\left| \int_{|y| |\xi| < 1} (\xi \cdot y) \int_0^1 \cos (t \xi \cdot y) ~dt~ a(y)J(y) ~dy\right|^2 d\xi \\
&& +I_{\sigma >1}N\int_{\bR^d} | \cF u (\xi)|^2
\left| \int_{|y| |\xi| < 1} (\xi \cdot y) \int_0^1 (\cos (t \xi \cdot y) -  1)~dt~ a(y)J(y) ~dy\right|^2 d\xi.
\end{eqnarray*}
Observe that by {\bf{H1}}, for any $t\in (0,1)$,
$$
\Psi(t|\xi|)=\int_{\bR^d}(1-\cos(t y\cdot \xi)) J(y)dy=t^{-d}\int_{\bR^d}(1-\cos(y\cdot \xi)J(t^{-1}y)dy\leq Nt^{\alpha_0}\Psi(|\xi|).
$$
Thus, if $\sigma>1$,
$$
\cI_{3} \leq N\int_{\bR^d}|\cF(u)|^2 \left(\int^1_0 \Psi(t|\xi|)dt\right)^2 \,d\xi\leq N\|\cA u\|^2_{L_2}.
$$
Also, if $\sigma>1$,
\begin{eqnarray*}
\tilde{\cI}_3
&\leq& (2\pi)^{-d} \int_{\bR^d} | \cF u (\xi)|^2
\left| \int_{|y| |\xi| < 1} (\xi \cdot y) \int_0^1 \cos (t \xi \cdot y) I_{|y|\geq 1} ~dt~ a(y)J(y) ~dy\right|^2 d\xi \\
&& +(2\pi)^{-d}\int_{\bR^d} | \cF u (\xi)|^2
\left| \int_{|y| |\xi| < 1} (\xi \cdot y) \int_0^1 (1-\cos (t \xi \cdot y))~dt~ a(y)J(y) ~dy\right|^2 d\xi \\
&\leq& N \int_{\bR^d} | \cF u (\xi)|^2 \left(\int_{|y|\geq 1} J(y)dy\right)^2\,d\xi+ N\int_{\bR^d}|\cF(u)|^2 \left(\int^1_0 \Psi(t|\xi|)dt\right)^2 \,d\xi\\
&\leq& N\|u\|^2_{\cH_2^\cA}.
\end{eqnarray*}
Thus  \eqref{4093} is proved if $\sigma>1$, and  \eqref{4083} and \eqref{40832} are obtained for general $u \in \cH_2^\cA$ owing to \eqref{4093}. Therefore (ii) is proved.

Now assume $\sigma\leq 1$. To estimate $\cI_{3}$ we use the Fubini's Theorem, the change of variable $|\xi|t y \to y$, {\bf H1}, {\bf H2},  and Lemma \ref{h3}
\begin{eqnarray*}
\cI_{3}
&\leq& N \int_{\bR^d} | \cF u (\xi)|^2 \\
&& \quad \quad \quad \times \left|  \int_0^1 t^{-d-1}|\xi|^{-d}\int_{|y|  < t } (\frac{\xi}{|\xi|} \cdot y)  \cos ( \frac{\xi}{|\xi|} \cdot y) a(\frac{y}{|\xi|t})J(\frac{y}{|\xi|t}) ~dy dt\right|^2 d\xi  \\
&\leq& N \int_{\bR^d} | \cF u (\xi)|^2\left| |\xi|^{-d} \int_0^1 t^{-d-1}\int_{|y|  < 1 } |y| J(\frac{y}{|\xi|t}) ~dy dt\right|^2 d\xi  \\
&\leq&  N \int_{\bR^d} | \cF u (\xi)|^2 \left| |\xi|^{-d}  \int_0^1 t^{\alpha_0-1}~dt\int_{|y| < 1 }  |y|J(y/|\xi|) ~dy\right|^2 d\xi\\
&\leq& N \int_{\bR^d} | \cF u (\xi)|^2 \left| |\xi|^{-d} j(1/|\xi|) \right|^2 d\xi
\leq N \|\cA u\|^2_{L_2}.
\end{eqnarray*}
Therefore the lemma is proved.
\end{proof}

Corollary \ref{unique} and Lemmas \ref{exis} and \ref{l2} easily prove Theorem \ref{main L_2}.

\section{Some H\"older estimates}\label{s:3}

In this section  obtain some   H\"older estimates for functions $u \in \cH_2^\cA \cap C_b^\infty$. The estimates will be used later  for the estimates of  the mean oscillation. Throughout this section we assume Assumption \ref{cancallation}  holds if $\sigma=1$.

\begin{lemma}
\label{convex}
For any $\alpha \in (0,1)$, $b \in \bR^d$, and a nonnegative measurable function $\cK(z)$, there exist $\eta_1, \eta_2 \in (0, 1/4)$, depending only on $\alpha$, such that
\begin{eqnarray}
\notag  &&\int_{\cC} [\left( |b +2z|^\alpha + |b - 2z|^\alpha - 2|b|^\alpha \right) \cK(z)]~dz\\
\label{31912}
  &\leq& - 2^{\alpha-3}  \alpha(1-\alpha) \int_{\cC} |b|^{\alpha-2} |z|^2 \cK(z) dz,
\end{eqnarray}
where
$$
\cC = \{|z| < \eta_1 |b| : |z \cdot b| \geq (1-\eta_2)|b||z|\}.
$$
\end{lemma}
\begin{proof}
We repeat the proof of Lemma 4.2 in \cite{dkim} with few minor changes.  Put $\eta(t) := b+ 2tz$ and $\varphi(t) :=|b+ 2t z|^\alpha  =  |\eta(t)|^\alpha$ for $z \in \cC$.
Then
\begin{eqnarray*}
\varphi '' (t)
&=&  \sum_{i,j=1}^d \left( \alpha(\alpha-2)(\eta_i(t))(\eta_j(t))|\eta(t)|^{\alpha-4}  +  I_{i=j}  \alpha |\eta(t)|^{\alpha-2}  \right)4z_i z_j\\
&=& 4 \alpha(\alpha-2)|\eta(t)|^{\alpha-4}|\eta(t) \cdot z|^2 + 4 \alpha |\eta(t)|^{\alpha-2} |z|^2 \\
&=&  4\alpha  |b+2tz|^{\alpha-4}[ (\alpha-2) |(b+2tz) \cdot z|^2 + |b+2tz|^2 |z|^2].
\end{eqnarray*}
For $t \in [-1,1]$ and $z \in \cC$, observer that,
$$
|b+2tz|^2 \leq (1+ 2\eta_1)^2 |b|^2
$$
and
\begin{eqnarray*}
|(b+2tz) \cdot z|
&=& | b \cdot z + 2t|z|^2 | \geq |b \cdot z|-2 |z|^2 \\
&\geq& (1-\eta_2)|b||z| - 2|z|^2 \geq (1- 2 \eta_1 - \eta_2) |z| |b|.
\end{eqnarray*}
Thus
\begin{eqnarray}
\label{3193}
\varphi '' (t) \leq 4 \alpha  |a+2tz|^{\alpha-4}[ (\alpha-2) (1-2\eta_1 -\eta_2)^2 + (1+2\eta_1)^2]|b|^2|z|^2.
\end{eqnarray}
Since $(1- 2\eta_1 -\eta_2)^2 \to 1$ and $(1+2\eta_1)^2 \to 1$ as $\eta_1, \eta_2 \downarrow 0$, one can choose sufficiently small $\eta_1,\eta_2 \in (0,1/4)$, depending only on
$\alpha \in (0,1)$, such that
$$
(\alpha -2 ) ( 1- 2\eta_1 - \eta_2)^2 + (1+ 2\eta_1)^2 \leq ( \alpha -1)/2.
$$
By combining this with \eqref{3193}
\begin{eqnarray}
\label{3194}\varphi''(t) \leq -2 \alpha(1-\alpha) |b+ 2tz|^{\alpha-4}|b|^2|z|^2.
\end{eqnarray}
Furthermore observe that
\begin{eqnarray*}
|b+2tz|^{\alpha -4} \geq (1+ 2\eta_1)^{\alpha-4} |b|^{\alpha-4} \geq 2^{\alpha-4} |b|^{\alpha-4}.
\end{eqnarray*}
Therefore, from \eqref{3194}
\begin{eqnarray*}
\varphi''(t) \leq - 2^{\alpha -3}  \alpha (1-\alpha) |b|^{\alpha-2}|z|^2, \quad t \in [-1, 1],~z \in \cC.
\end{eqnarray*}
In addition to this, to prove (\ref{31912}), it is enough to    use the fact that there exists $t_0 \in (-1,1)$ satisfying
$$
\varphi(1) + \varphi(-1) -2 \varphi(0) = \varphi''(t_0),
$$
which can be shown by the mean value theorem.
The lemma is proved.
\end{proof}


\begin{theorem} \label{holder}

Let $R>0, \lambda \geq 0$ and {\bf H1} hold. Suppose $f  \in L_\infty(B_1)$ and $u, \tilde u \in C_b^2( B_R) \cap L_1(\bR^d,w_R)$, where
$w_R(x)= \frac{1}{ 1/j(R)+ 1/J(x/2)}$. Also assume
\begin{eqnarray}
\label{3111}Lu - \lambda u =f,  \quad  \quad \tilde L \tilde u - \lambda \tilde u = f \quad \text{in}~\,\,B_R.
\end{eqnarray}

(i)  For any $\alpha \in (0, \min\{1, \alpha_0 \})$ and $0<r<R$, it holds that
\begin{eqnarray}
\notag [u]_{C^\alpha(B_r)}&\leq& Nr_1^{-\alpha}\|u\|_{L_\infty(B_R)}\\
\notag &&+ N \frac{ \|u \|_{L_\infty(B_R)}}{j(r_1)r_1^{d+\alpha}} \Big( r_1^{-2}  \int_{B_{r_1}} |z|^2 J(z)~dz
+ I_{\sigma <1} r_1^{-1} \int_{B_{r_1}} |z|   J(z)~dz  \Big) \\
 &&+  N\Big(\frac{1}{r^{d+\alpha}_1j(R)}\|u\|_{L_1(\bR^d,w_R)} +     \frac{1}{j(r_1)r_1^{d+\alpha}} \osc_{B_R} f \Big),
                                \label{3112}
\end{eqnarray}
where $r_1 = (R-r)/2$ and  $N=N(d,\nu,\Lambda,\kappa_1, \alpha_0, \alpha)$.

Consequently, if  {\bf{H2}}  is additionally assumed, then
\begin{equation}
           \label{5.15.1}
    [u]_{C^\alpha(B_r)}\leq N \left(r_1^{-\alpha}\|u\|_{L_\infty(B_R)}+ \frac{1}{r_1^{d+\alpha}j(R)}\|u\|_{L_1(\bR^d,w_R)}+  \frac{\osc_{B_R} f }{j(r_1)r_1^{d+\alpha}}\right).
    \end{equation}

(ii) In addition to {\bf{H1}}, let one of {\bf{H3}}(ii)- {\bf{H3}}(iv) hold. Then (\ref{3112}) holds for $\tilde{u}$.
 Consequently, if  {\bf{H2}}  additionally holds,  (\ref{5.15.1}) holds for $\tilde{u}$.

\end{theorem}

\begin{proof}
We adopt the method  used in \cite{dkim} (cf. \cite{bci}).
Assume that $u$ is not identically zero in $B_r$.
Set
$$r_1= (R-r)/2, \quad r_2 = (R+r)/2, \quad w(t,x) = I_{B_R}(x) u(t,x).
$$
 For $ x\in B_{r_2}$, $u(x)=v(x)$ and
$\nabla u(x) = \nabla w(x)$.  Thus
\begin{eqnarray*}
Lu(x)= Lw(t,x) + \int_{|z| \geq r_1} \left( u(t,x+z) - w(t,x+z) \right) a(z)J(z) dz.
\end{eqnarray*}
So in $ B_{r_2}$
\begin{eqnarray}
\label{4261}Lw(x) -\lambda w = g(x) + f(x),
\end{eqnarray}
where
$$
g(x) =- \int_{|z| \geq r_1} \left( u(x+z) - w(x+z) \right) a(z)J(z) dz.
$$
Note that by {\bf H1}
\begin{eqnarray}\label{3142}
\|g\|_{L_\infty (B_R)}   \leq N \frac{j(r_1)}{j(R)} \|u\|_{L_1(\bR^d,w_R)},
\end{eqnarray}
where $N=N(d, \Lambda)$.
Indeed, this comes from the fact that for all $|z| \geq r_1 $, $x \in B_R$, and $|x+z| \leq R$
\begin{eqnarray*}
|j(z)| \leq  Nj(r_1) \leq \frac{j(r_1)}{j(R)} \cdot \frac{N}{1/j(R) + 1/j(|x+z|/2)}.
\end{eqnarray*}
For $x_0 \in B_r$ and $\alpha \in (0, \min\{1, \alpha_0 \})$, we define
$$
M(x,y) := w(x) - w(y) - C|x-y|^\alpha - 8r_1^{-2}   \|u\|_{L_\infty(B_R)}|x-x_0|^2,
$$
where $C$ is a positive constant which will be chosen later so that it is  independent of $x_0$ and
\begin{eqnarray}
\label{3121}\sup_{ x,y \in \bR^d} M(x,y) \leq 0.
\end{eqnarray}
For $x \in \bR^d \setminus B_{r_1/2}(x_0)$,
\begin{eqnarray}
\label{3122} w(x) - w(y) \leq 2  \|u\|_{L_\infty(B_R)} \leq 8r_1^{-2}   \|u\|_{L_\infty(B_R)} |x-x_0|^2.
\end{eqnarray}
This shows
$$
M(x,y) \leq 0, \quad x \in \bR^d \setminus B_{r_1/2}(x_0).
$$
Assume that there exist $x,y \in \bR^d$ such that $M(x,y) >0$. We will get the contradiction by choosing an appropriate constant $C$.
Due to \eqref{3122}, $x \in B_{r_1/2}(x_0) $. Moreover
\begin{eqnarray*}
 w(x) -w(y) > C|x-y|^{\alpha},
\end{eqnarray*}
which implies
\begin{eqnarray} \label{3123}
|x-y|^\alpha < \frac{2  \|u\|_{L_\infty(B_R)} }{C}.
\end{eqnarray}
If we take $C$ large enough so that $C \geq 2 (r_1/2)^{-\alpha}  \|u\|_{L_\infty(B_R)}$, then
$$
y \in B_{r+r_1}.
$$
Therefore, there exist $\bar x, \bar y \in B_{r+r_1}$ satisfying
$$
\sup_{x,y \in \bR^d}M(x,y) = M(\bar x, \bar y) >0.
$$
Moreover, from \eqref{4261}
\begin{eqnarray}
- 2 \| g \|_{L_\infty(B_R)} -  \osc_{B_R} f
 &\leq& (Lw( \bar x) -\lambda w (\bar x))  -  ( Lw( \bar y)-\lambda w(\bar y)) \nonumber \\
 &=& (Lw( \bar x) -Lw( \bar y))+ \lambda( w (\bar y) - w(\bar x))  \nonumber\\
 & \leq &  Lw( \bar x)- Lw( \bar y):= \cI. \label{3131}
\end{eqnarray}
Put $K(z):=a(z)J(z)$ and
$$
K_1(z) :=  K(z) \wedge K(-z), \quad K_2(z) := K(z) - K_1(z).
$$
By $L_1$ and $L_2$, respectively, we denote the operators with kernels $K_1$ and $K_2$.
Then
$$
\cI=\cI_1+\cI_2,
$$
where
$$
\cI_1 :=  L_1w( \bar x)- L_1w( \bar y) \quad \text{and} \quad  \cI_2 :=  L_2w( \bar x)- L_2w( \bar y).
$$
Since $K_1$ is symmetric (i.e. $K_1(z)=K_1(-z)$),
$$
\cI_1 = \frac{1}{2} \int_{\bR^d} \cJ(\bar x, \bar y, z) K_1(z) dz,
$$
where
$$
\cJ(\bar x, \bar y, z ) = w(\bar x +z) + w(\bar x - z) - 2 w(\bar x) - w(\bar y + z ) - w(\bar y - z) +2w(\bar y).
$$
Also, since $M(x,y)$ attains its maximum at $(\bar x, \bar y)$,
\begin{eqnarray}
\notag &&  w( \bar x +z)-w(\bar y + z) -C|\bar x - \bar y|^\alpha - 8r_1^{-2}   \|u\|_{L_\infty(B_R)}|\bar x +z-x_0|^2 \\
\label{3133} &\leq&  w( \bar x )-w(\bar y ) -C|\bar x - \bar y|^\alpha - 8r_1^{-2}    \|u\|_{L_\infty(B_R)}|\bar x -x_0|^2
\end{eqnarray}
and
\begin{eqnarray}
\notag &&  w( \bar x-z)-w(\bar y -z) - C|\bar x - \bar y|^\alpha - 8r_1^{-2}   \|u\|_{L_\infty(B_R)}|\bar x - z-x_0|^2 \\
\label{3134} &\leq&  w( \bar x )-w(\bar y ) -C|\bar x - \bar y|^\alpha -8r_1^{-2}   \|u\|_{L_\infty(B_R)}|\bar x -x_0|^2
\end{eqnarray}
for all $z \in \bR^d$. By combining these two inequalities,
\begin{eqnarray}
                                \label{3135}
\cJ(\bar x, \bar y, z)
\leq 8r_1^{-2}    \|u\|_{L_\infty(B_R)}
  \left( |\bar x + z-x_0|^2 +|\bar x -z -x_0|^2- 2 |\bar x-x_0|^2 \right).
\end{eqnarray}
Similarly,
\begin{eqnarray*}
\notag &&  w( \bar x+z)-w(\bar y -z) -C|\bar x - \bar y + 2 z|^\alpha - 8r_1^{-2}   \|u\|_{L_\infty(B_R)}|\bar x+ z -x_0|^2 \\
&\leq&  w( \bar x )-w(\bar y ) -C|\bar x - \bar y|^\alpha - 8r_1^{-2}    \|u\|_{L_\infty(B_R)}|\bar x -x_0|^2,
\end{eqnarray*}
\begin{eqnarray*}
\notag &&  w( \bar x-z)-w(\bar y+z) -C|\bar x - \bar y - 2 z|^\alpha -8r_1^{-2}   \|u\|_{L_\infty(B_R)}|\bar x -z -x_0|^2 \\
&\leq&  w( \bar x )-w(\bar y ) -C|\bar x - \bar y|^\alpha - 8r_1^{-2}    \|u\|_{L_\infty(B_R)}|\bar x -x_0|^2.
\end{eqnarray*}
It follows that, for any $z \in \bR^d$,
\begin{eqnarray}
 \cJ(\bar x, \bar y, z) &\leq& C\left(|\bar x - \bar y + 2z|^\alpha + | \bar x - \bar y - 2z|^\alpha - 2 |\bar x - \bar y|^\alpha \right) \label{3136} \\
&&+8r_1^{-2}   \|u\|_{L_\infty(B_R)}
 \left(|\bar x +z -x_0|^2 + |\bar x -z-x_0|^2 - 2|\bar x -x_0|^2\right). \nonumber
\end{eqnarray}
Put $b= \bar x - \bar y$. Since $(\bar x, \bar y)$ satisfy \eqref{3123}, $|b| < r_1/2$ if $C \geq 2 (r_1/2)^{-\alpha}  \|u\|_{L_\infty(B_R)}$.
 Also set for $\eta_1, \eta_2 \in (0, 1/4)$ specified in Lemma \ref{convex},
$$
\cC = \{|z| < \eta_1 |b| : |z \cdot b| \geq (1-\eta_2)|b||z|\}.
$$
Then
\begin{eqnarray}
\notag 2\cI_1 &=& \int_{|z| \geq r_1/2} \cJ(\bar x, \bar y, z) K_1(z)~dz + \int_{B_{r_1/2} \setminus \cC} \cJ(\bar x, \bar y, z) K_1(z)~dz \\
\label{3137}&&+ \int_{\cC} \cJ(\bar x, \bar y, z) K_1(z) ~dz := \cI_{11}+ \cI_{12}+\cI_{13}.
\end{eqnarray}
Note that by {\bf H1},
$$
\cI_{11} \leq Nj(r_1/2)r_1^d  \|u\|_{L_\infty(B_R)}.
$$
Indeed,
\begin{eqnarray*}
\cI_{11}
&\leq& N \|u\|_{L_\infty(B_R)}\int_{|z| \geq r_1/2} J(z)~dz \\
&\leq& N r_1^d\|u\|_{L_\infty(B_R)}\int_{|z| \geq 1} J(r_1z/2)~dz \\
&\leq& N j(r_1/2)r_1^d\|u\|_{L_\infty(B_R)}\int_{|z| \geq 1} |z|^{-d-\alpha_0}~dz.
\end{eqnarray*}

On the other hand from \eqref{3135}, it follows that
\begin{eqnarray*}
\cI_{12} &\leq& 8r_1^{-2}   \|u\|_{L_\infty(B_R)}   \int_{B_{r_1/2} \setminus \cC} \left(|\bar x +z  -x_0|^2 + |\bar x -z-x_0|^2 - 2|\bar x -x_0|^2\right)
K_1(z)~dz \\
&\leq& Nr_1^{-2}   \|u\|_{L_\infty(B_R)} \int_{B_{r_1/2}} |z|^2 J(z)~dz.
\end{eqnarray*}

Next using \eqref{3136} we obtain
\begin{eqnarray*}
 &&\cI_{13}
 \leq  C \int_{\cC}  \left( |\bar x - \bar y +2z|^\alpha + |\bar x - \bar y - 2z|^\alpha - 2|\bar x - \bar y|^\alpha \right)K_1(z)~dz \\
&&+8r_1^{-2}   \|u\|_{L_\infty(B_R)}\int_{\cC}  \left(|\bar x +z -x_0|^2 + |\bar x -z-x_0|^2 - 2|\bar x -x_0|^2\right)K_1(z)~dz\\
&&:= \cI_{131} + \cI_{132}.
\end{eqnarray*}
The term $\cI_{132}$ is again bounded by
$$
N r_1^{-2}\|u\|_{L_\infty(B_R)} \int_{B_{r_1/2}} |z|^2 J(z)~dz .
$$
Furthermore, from lemma \ref{convex}
$$
\cI_{131} \leq - 2^{\alpha-3} C \alpha(1-\alpha) \int_{\cC} |b|^{\alpha-2} |z|^2 K_1(z) dz.
$$
Combining all these facts above, we obtain
\begin{eqnarray}
\notag \cI_1
&\leq& N \|u( \cdot)\|_{L_\infty(B_R)} \left(j(r_1/2)r_1^d  +r_1^{-2} \int_{B_{r_1/2}} |z|^2 J(z)~dz \right)   \\
\label{3143} &&- 2^{\alpha-3} C \alpha(1-\alpha) \int_{\cC} |b|^{\alpha-2} |z|^2 K_1(z) dz.
\end{eqnarray}
For $\cI_2$, we first    consider {\bf the case $\sigma <1$}. In this case,
\begin{eqnarray*}
\cI_2
&=& \int_{|z| \geq r_1/2} \left( w(\bar x +z) - w(\bar x) - w (\bar y +z) + w(\bar y)\right) K_2(z)~dz\\
&&+~ \int_{B_{r_1/2}} \left( w(\bar x +z) - w(\bar x) - w (\bar y +z) + w(\bar y)\right) K_2(z)~dz:=\cI_{21}+\cI_{22}.
\end{eqnarray*}
Analogously to $\cI_{11}$, we bound $\cI_{21}$ by $N j(r_1/2)r_1^d\|u\|_{L_\infty(B_R)}$.
For the other term $\cI_{22}$, since $|\bar x -x_0| < r_1/2$, from \eqref{3133}
\begin{eqnarray*}
\cI_{22} &\leq& Nr_1^{-2}   \|u\|_{L_\infty(B_R)} \int_{B_{r_1/2}}   \left( |\bar x+z -x_0|^2-|\bar x -x_0|^2 \right) K_2(z)~dz \\
&\leq& Nr_1^{-2}   \|u\|_{L_\infty(B_R)}  \int_{B_{r_1/2}}\left(|z|^2 + 2|z| |\bar x - x_0|\right) J(z)~dz\\
&\leq& Nr_1^{-1}   \|u\|_{L_\infty(B_R)}  \int_{B_{r_1/2}} |z|   J(z)~dz.
\end{eqnarray*}
So
\begin{eqnarray}
                             \label{3141}
 \cI_2 \leq N\|u \|_{L_\infty (B_R)} \Big(\,j(r_1/2)r_1^d
 + r_1^{-1}  \int_{B_{r_1/2}} |z|   J(z)~dz \Big).
\end{eqnarray}
By combining \eqref{3142}, \eqref{3131}, \eqref{3143} and \eqref{3141},
\begin{eqnarray*}
0&\leq& N_1 \Big( \osc_{B_R} f + \frac{j(r_1)}{j(R)}\|u\|_{L_1(\bR^d,w_R)} \\
&& \quad \quad +  \|u \|_{L_\infty(B_R)} \big[ j(r_1/2)r_1^d
+r_1^{-1}  \int_{B_{r_1/2}} |z|   J(z)~dz   \big]  \Big)\\
&&- 2^{\alpha-3} C \alpha(1-\alpha) \int_{\cC} |b|^{\alpha-2} |z|^2 K_1(z)~dz .
\end{eqnarray*}
Thus, if $C \geq C_1:=2(r_1/2)^{-\alpha} \|u\|_{L_\infty(B_R)}$ and
\begin{eqnarray*}
C &\geq& C_2:=N_1 C_3 \Big( \osc_{B_R} f +  \frac{j(r_1)}{j(R)}\|u\|_{L_1(\bR^d,w_R)} \\
&&
\quad \quad + \|u \|_{L_\infty(B_R)} \Big[ j(r_1/2)r_1^d
+ r_1^{-1}  \int_{B_{r_1/2}} |z|   J(z)~dz \Big]   \Big),
\end{eqnarray*}
then
\begin{eqnarray*}
0 &\leq& N_1\Big( \osc_{B_R} f
 + \frac{j(r_1)}{j(R)}\|u\|_{L_1(\bR^d,w_R)}\\
 &&\quad \quad
 + \|u \|_{L_\infty(B_R)} \Big[ j(r_1/2)r_1^d
 +r_1^{-1}  \int_{B_{r_1/2}} |z|   J(z)~dz   \Big]\Big)  \\
 && \times \Big(1- C_3  2^{\alpha-3} \alpha(1-\alpha) \int_{\cC} |b|^{\alpha-2} |z|^2 K_1(z)~dz\Big)\\
 &:=&(1-C_3 C_4(b)).
\end{eqnarray*}
If we take $C_3$ so that $C_3  = 1/C_5$ for a $C_5=C_5(r_1,\alpha) < C_4(b)$ which does not depend on $b$ and will be chosen below, we get the contradiction.
 To select $C_5$, observe that with {\bf H1} and the fact $|b| \leq r_1/2$
\begin{eqnarray*}
 C_4(b) &=&2^{\alpha-3} \alpha(1-\alpha) \int_{\cC} |b|^{\alpha-2} |z|^2 K_1(z) dz\\
&\geq&   \nu2^{\alpha-3} \alpha(1-\alpha) \int_{\cC} |b|^{\alpha-2} |z|^2 J(z) dz \\
&\geq&  \kappa_1^{-1} \nu 2^{\alpha-3} \alpha(1-\alpha)j(\eta_1|b|)  \int_{\cC} |b|^{\alpha-2} |z|^2  dz \\
&\geq&   \kappa_1^{-1} \nu 2^{\alpha-3} \alpha(1-\alpha)j(\eta_1|b|)|b|^{\alpha-2} |\eta_1 b|^{d+2} \int_{\cC_{\eta_2}}  |z|^2  dz \\
&\geq&  \kappa_1^{-1} \nu \eta_1^{d+2}2^{\alpha-3} \alpha(1-\alpha)j(|b|)|b|^{d+\alpha}  \int_{\cC_{\eta_2}}  |z|^2  dz\\
&\geq&    \kappa_1^{-2} \nu j(r_1/2)(r_1/2)^{d+\alpha}  \eta_1^{d+2}2^{\alpha-3} \alpha(1-\alpha) \int_{\cC_{\eta_2}}  |z|^2  dz\\
\notag &=&j(r_1/2)r_1^{d+\alpha}N(\alpha,\eta_1,\eta_2):=C_5,
\end{eqnarray*}
where $\cC = \{|z| < \eta_1 |b| : |z \cdot b| \geq (1-\eta_2)|b||z|\}$ and $\cC_{\eta_2} = \{|z| < 1 : \frac{|z \cdot b|}{|b||z|} \geq (1-\eta_2)\}$.
Therefore, \eqref{3121} holds with $C=C_1+C_2$.
Since $C$ is independent of $x_0$, (\ref{3112}) is proved.

Next we consider {\bf the case $\sigma = 1$}. Note that, because $K_1$ is symmetric, both $K_1$ and $K_2$ satisfy \eqref{cancel}.
Therefore, we can replace $1_{B_1}$ with $I_{B_{r_1}}$ in the definition of $L_2$, and get $\cI_2= \cI_{21}+\cI_{22}$, where
\begin{eqnarray*}
\cI_{21} = \int_{|z| \geq r_1/2} \left( w(\bar x + z) - w(\bar x) - w(\bar y + z) + w (\bar y) \right)K_2(z)~dz,
\end{eqnarray*}
\begin{eqnarray*}
\cI_{22} = \int_{B_{r_1/2}} \left( w(\bar x +z) -w (\bar x) - w(\bar y +z) + w(\bar y) - z \cdot (\nabla w (\bar x)- \nabla w(\bar y)) \right) K_2(z)~dz.
\end{eqnarray*}
$\cI_{21}$ is already estimated in the previous case. Thus we only consider   $\cI_{22}$.
Since $M(x,y)$ attains its maximum at the interior point $(\bar x, \bar y)$,  we have $\nabla_x M(\cdot,\bar{y})(\bar{x})=0$, $\nabla_y M(\bar{x},\cdot)(\bar{y})=0$, and therefore
\begin{eqnarray}
\label{3191} && \nabla w (\bar x) -\nabla w(\bar y) =  16 r_1^{-2} \|u\|_{L_\infty(B_R)}(\bar x -x_0).
\end{eqnarray}
We use \eqref{3133} and \eqref{3191} to get
\begin{eqnarray*}
\cI_{22}
&\leq& 8r_1^{-2} \|u\|_{L_\infty(B_R)} \int_{B_{r_1/2}} |z|^2 K_2(z)~dz\\
&\leq& 8r_1^{-2}  \int_{B_{r_1/2}} |z|^2 J(z)~dz \|u\|_{L_\infty (B_R)}.
\end{eqnarray*}
Therefore, (\ref{3112}) is   proved  following the argument in the case $\sigma <1$.

Finally, let {\bf $\sigma >1$}. Now  we have $\cI_2 = \cI_{21} + \cI_{22}$, where
$$
\cI_{21} = \int_{|z| \geq r_1/2} [   w(\bar x + z) -  w( \bar x) - w (\bar y + z) + w (\bar y) - z \cdot (\nabla w (\bar x) - \nabla w (\bar y))  ]K_2(z)~dz,
$$
$$
\cI_{22} = \int_{B_{r_1/2}}  [  w(\bar x + z) -  w( \bar x) - w (\bar y + z) + w (\bar y) - z \cdot (\nabla w (\bar x) - \nabla w (\bar y))  ] K_2(z)~dz.
$$
Since $\sigma >1 $, $|\bar x - x_0| < r_1/2$, by \eqref{3191} and {\bf H1}
\begin{eqnarray*}
\cI_{21}
&\leq& \int_{|z| \geq r_1/2} [ 4 \|u\|_{L_\infty (B_R)} + 4 (r_1/2)^{-1} \|u\|_{L_\infty(B_R)}|z|  ] K_2(z)~dz~\\
&\leq& N r_1^d  j(r_1/2) \| u \|_{ L_\infty (B_R)}.
\end{eqnarray*}
For $\cI_{22}$, we apply \eqref{3133} and \eqref{3191} to get
\begin{eqnarray*}
\cI_{22}
&\leq& N r_1^{-2}  \|u\|_{L_\infty(B_R)} \int_{B_{r_1/2}} |z|^2 J(z)~dz.
\end{eqnarray*}
So we again argue as in the first case to get the contradiction. Hence (i) is proved.

The proof of (ii) is quite similar to that of (i).
Denote the counter parts of $w$ and $g$ by $\tilde w $ and $\tilde g$, respectively. Also we introduce $\cI_1$ and $\cI_2$ similarly. That is $\cI_1$ is same as before, and $\cI_2$ is given by
 \begin{eqnarray*}
\cI_2&=& \int_{|z| \geq r_1/2} \Big[ \tilde w(\bar x +z) - \tilde w(\bar x) - \tilde w (\bar y +z) + \tilde w(\bar y)\\
&& \quad - I_{B_1}(z) z \cdot \nabla (\tilde w(\bar x) - \tilde w(\bar y))\Big] K_2(z)~dz \\
&&+ \int_{B_{r_1/2}}\Big[ \tilde w(\bar x +z) - \tilde w(\bar x) - \tilde w(\bar y +z) + \tilde w(\bar y) \\
&& \quad -I_{B_1}(z) z \cdot \nabla (\tilde w(\bar x) - \tilde w (\bar y))\Big] K_2(z)~dz \\
&:=&\cI_{21}+\cI_{22}.
\end{eqnarray*}
All of the differences are as follows. If $r_1/2 \geq 1$, then by using \eqref{3133} and \eqref{3191},
\begin{eqnarray*}
\cI_{22}
&\leq& Nr_1^{-2} \|\tilde u\|_{L_\infty(B_R)} \Big[\int_{B_1} |z|^2  K_2(z)~dz
 + \int_{1 \leq |z| \leq r_1/2} (|z|^2 +(\bar x-x_0)\cdot z) K_2(z)~dz\Big] \\
&\leq& NI_{\sigma<1}r_1^{-1}\| \tilde{u}\|_{L_\infty (B_R)}  \int_{B_{r_1/2}} |z| J(z)~dz\\
&&+N I_{\sigma>1} r_1^{-2}\| \tilde{u}\|_{L_\infty (B_R)}  \int_{B_{r_1/2}} |z|^2 J(z)~dz.
\end{eqnarray*}
In the above, we also used $\int_{1\leq |z|\leq r_1/2} z^i K_2(z)dz=0$ if $\sigma>1$ (due to {\bf{H3}}(iv)).

Let $\sigma <1$ and $r_1/2 <1$. If {\bf{H3}}(ii) hold, then by (\ref{ma}),
 \begin{eqnarray*}
 \cI_{21}&\leq& N\|u\|_{L_\infty(B_R)} \int_{|z| \geq r_1/2}  J(z)~dz \\
&=& Nr_1^d\int_{|z|\geq 1}J(r_1z/2)dz\leq Nj(r_1/2) r_1^d \|u\|_{L_\infty(B_R)}.
 \end{eqnarray*}
 Also, if {\bf{H3}}(iii) holds, then by using \eqref{3191},
\begin{eqnarray*}
\cI_{21}
&\leq& \|u\|_{L_\infty(B_R)} \int_{|z| \geq r_1/2} [1+ 8r_1^{-1}|z|  ] K_2(z)~dz \\
&\leq& Nj(r_1/2) r_1^d \|u\|_{L_\infty(B_R)}.
\end{eqnarray*}
This completes the proof of the theorem.
\end{proof}

We remove $\sup_{B_R} u$ on the right hand side of \eqref{5.15.1} in the following corollary. Recall  $w_R(x)= \frac{1}{ 1/j(R)+ 1/J(x/2)}$.

\begin{corollary}
\label{cor1}
Suppose that {\bf H1} and {\bf{H2}} hold.
Let $\lambda \geq 0$, $f \in L_\infty(B_1)$, and $u, \tilde u \in C_b^2( B_R) \cap L_1(\bR^d,w_R)$ satisfy
\begin{eqnarray}
                                     \label{3111}
Lu - \lambda u =f, \quad  \quad\tilde L \tilde u - \lambda \tilde u = f \quad \quad \text{in}\quad B_R.
\end{eqnarray}
 (i)  For any $\alpha \in (0, \min\{1, \alpha_0 \})$, it holds that
\begin{eqnarray}
                                \label{eqn 5.25.7}
[u]_{C^\alpha(B_{R/2})} \leq  \frac{N}{j(R)R^{d+\alpha}} \left(\|u\|_{L_1(\bR^d,w_R)} + \osc_{B_R}f\right),
\end{eqnarray}
where $N=N(d, \nu, \Lambda, \kappa_1, \alpha_0,\alpha)$.

(ii) If one of  {\bf H3} (ii)-(iv) is additionally assumed, then (\ref{eqn 5.25.7}) holds for $\tilde{u}$.

\end{corollary}
\begin{proof}
 For $n=1,2,\ldots$, set
$$
r_n := R(1- 2^{-n}).
$$
Observe that $(r_{n+1} - r_n)/2 =R 2^{-n-2}\leq R$ and by {\bf H1}
\begin{eqnarray*}
\frac{1}{j(r_{n+1})} \|u\|_{L_1(\bR^d,w_{r_{n+1}})}
&\leq&  \left( \int_{|z| < 2R}u(z) ~dz  + \frac{1}{j(r_{n+1})} \int_{|z| \geq 2R} u(z) j(z/2) ~dz\right)\\
&\leq&  N \left( \int_{|z| < 2R}u(z) ~dz  + \frac{1}{j(R)} \int_{|z| \geq 2R} u(z) j(z/2) ~dz\right)\\
&\leq& N  \frac{ 1}{j(R)} \int_{\bR^d}u(z) w_R(z)~dz.
\end{eqnarray*}
 Then by
Theorem \ref{holder} (i) and {\bf H1},
\begin{eqnarray}
[u]_{C^\alpha(B_{r_n})}
 &\leq& N   R^{-\alpha}2^{ \alpha n}  \sup_{B_{r_{n+1}}}|u|  \nonumber\\
 && +N
 \frac{2^{(d+\alpha)n}}{j(R2^{-n-2})R^{d+\alpha}} \left(\frac{j(R2^{-n-2})}{j(r_{n+1})} \|u\|_{L_1(\bR^d,w_{r_{n+1}})} +  \osc_{B_{r_{n+1}}} f \right) \nonumber \\
 &\leq& N \Big[ R^{-\alpha}2^{ \alpha n}  \sup_{B_{r_{n+1}}}|u|
 + \frac{2^{(d+\alpha)n}}{j(R)R^{d+\alpha}} \Big( \|u\|_{L_1(\bR^d,w_R)} +  \osc_{B_R} f \Big)\Big]. \nonumber \\
 \label{3221}
\end{eqnarray}
In order to estimate the term $\sup_{B_{r_{n+1}}}|u|$ above,  we use the following :
\begin{equation}
                  \label{3222}
\sup_{B_{r_{n+1}}}|u| \leq (\varepsilon {r_{n+1}})^\alpha [u]_{C^\alpha(r_{n+1})} + N (\varepsilon r_{n+1} )^{-d} \| u\|_{L_1(B_{r_{n+1}})} , \quad \varepsilon \in (0,1).
\end{equation}
Actually  this inequality can be easily obtained as follows. For all $\varepsilon \in (0,1)$, $x \in B_{r_{n+1}}$ and $y \in B_{r_{n+1}} \cap B_{\varepsilon r_{n+1}}(x)$,
\begin{eqnarray*}
&&|B_{r_{n+1}} \cap B_{\varepsilon r_{n+1}}(x)| \cdot |u(x)| \\
&\leq& \int_{B_{r_{n+1}} \cap B_{\varepsilon r_{n+1}}(x)}\left(|u(x)-u(y)| + |u(y)| \right)~dy \\
&\leq& |B_{r_{n+1}} \cap B_{\varepsilon r_{n+1}}(x)|\cdot (\varepsilon {r_{n+1}})^\alpha [u]_{C^\alpha(B_{r_{n+1}})} + \int_{B_{r_{n+1}} \cap B_{\varepsilon r_{n+1}}(x)}|u(y)|~dy.
\end{eqnarray*}
 Now it is enough to note that $|B_{r_{n+1}} \cap B_{\varepsilon r_{n+1}}(x)| \sim (\varepsilon r_{n+1})^d$ because $\varepsilon \in (0,1)$ and $x \in B_{r_{n+1}}$.

Take $N$ from (\ref{3221}) and define $\varepsilon$ so that
$$
\varepsilon^\alpha = N^{-1} 2^{- \alpha n}2^{-3d}.
$$
Then by combining \eqref{3221} and \eqref{3222},
\begin{eqnarray}
\notag  [u]_{C^\alpha(B_{r_n})} &\leq&    2^{-3d}[u]_{C^\alpha(B_{r_{n+1}})}+N R^{-d-\alpha}2^{2dn}\| u\|_{L_1(B_{r_{n+1}})} \\
\notag &&+ N \frac{2^{(d+\alpha)n}}{j(R)R^{d+\alpha}} ( \|u\|_{L_1(\bR^d,w_R)}+ \osc_{B_R} f)\\
\notag &\leq&    2^{-3d}[u]_{C^\alpha(B_{r_{n+1}})} + N R^{-d-\alpha}  2^{2dn}\| u\|_{L_1(B_{r_{n+1}})} \\
&&+ N \frac{2^{2dn}}{j(R)R^{d+\alpha}}(  \|u\|_{L_1(\bR^d,w_R)}+ \osc_{B_R} f).
                                                                                        \label{3223}
\end{eqnarray}
Multiply both sides of \eqref{3223} by $  2^{-3dn}$ and take the  sum over $n$ to get
\begin{eqnarray*}
&& \sum_{n=1}^\infty   2^{-3dn} [u]_{C^\alpha(B_{r_n})} \\
&\leq&  \sum_{n=1}^\infty  2^{-3d(n+1)}[u]_{C^\alpha(B_{r_{n+1}})} + N\sum_{n=1}^\infty 2^{-dn}R^{-d-\alpha}\| u\|_{L_1(B_{r_{n+1}})} \\
&&+ N\Big(\sum_{n=1}^\infty 2^{-dn} \Big)\frac{1}{j(R)R^{d+\alpha}} (\|u\|_{L_1(\bR^d,w_R)} + N \osc_{B_R} f).
\end{eqnarray*}
Since $[u]_{C^{\alpha}(B_{r_n})}\leq [u]_{C^{\alpha}(B_R)}<\infty$ and by {\bf H1}
\begin{eqnarray*}
\| u\|_{L_1(B_{r_{n+1}})}
\leq \| u\|_{L_1(B_R)}= \frac{j(R)}{j(R)}\| u\|_{L_1(B_R)} \leq \frac{N}{j(R)} \|u\|_{L_1(\bR^d,w_R)},
\end{eqnarray*}
(i) is proved.

(ii) is proved similarly by  following the proof of (i) with Theorem \ref{holder} (ii).
\end{proof}

\section{Some sharp function and maximal function estimates} \label{s:4}

For $g \in L_{1,\loc}(\bR^d)$, the maximal function and sharp function are defined as follows :
$$
\cM g (x) := \sup_{r>0} \aint_{B_r(x)} |g(y)|~dy :=\sup_{r>0} \frac{1}{|B_r(x)|}\int_{B_r(x)} |g(y)|~dy,
$$
and
$$
g^{\#} (x) := \sup_{r>0} \aint_{B_r(x)} |g(y)-(g)_{B_r(x)}|~dy :=\sup_{r>0} \frac{1}{|B_r(x)|}\int_{B_r(x)} |g(y)-(g)_{B_r(x)}|~dy,
$$
where $(g)_{B_r(x)} = \frac{1}{|B_r(x)|}\int_{B_r(x)} g(y) ~dy$ the average of $g$ on $B_r(x)$.

\begin{lemma}
\label{holderle}
Suppose that {\bf H1} and {\bf H2} hold.
Let $\lambda \geq 0$, $R>0$, $f \in C_0^\infty$,  and $f=0$ in $B_{2R}$. Assume that $u, \tilde u \in H_2^{\cA} \cap C_b^\infty$ satisfy
\begin{equation}
                                        \label{3211}
Lu -\lambda u =f, \quad \quad
\tilde L \tilde u -\lambda \tilde u =f.
\end{equation}

(i) Then for all $\alpha \in (0, \min\{1,\alpha_0\})$,
\begin{equation}
                                           \label{3212}
[u]_{C^{\alpha}(B_{R/2})} \leq N R^{-\alpha}\sum_{k=1}^\infty 2^{- \alpha_0 k}(|u|)_{B_{2^kR}},
\end{equation}
\begin{equation}
                           \label{3213}
[\cA u]_{C^{\alpha}(B_{R/2})} \leq N R^{-\alpha} \left( \sum_{k=1}^\infty 2^{-\alpha_0 k }(|\cA u|)_{B_{2^kR}}  + \cM f(0) \right),
\end{equation}
where $N$ depends only on $d,\nu,\Lambda,\kappa_1, \kappa_2,\alpha_0$, and $\alpha$.

(ii) If one of  {\bf H3}(ii)-(iv) is additionally assumed, then \eqref{3212} and \eqref{3213} hold for $\tilde u$.

\end{lemma}
\begin{proof}
By Corollary \ref{cor1} and the assumption that $f=0$ in $B_{2R}$,
\begin{eqnarray}
[u]_{C^\alpha(B_{R/2})} \leq N \frac{1}{j(R)R^{d+\alpha}}\|u\|_{L_1(\bR^d,w_R)}.
\end{eqnarray}
Set
$$
B_{(0)} =B_{R}, \quad B_{(k)} = B_{2^kR} \setminus B_{2^{k-1}R}, \quad  k \geq 1.
$$
Observe that
\begin{eqnarray*}
\|u\|_{L_1(\bR^d,w_R)}
&=& \int_{\bR^d} |u(y)| \frac{1}{1/j(R)+ 1/j(|y|/2)} ~dy \\
&=& \sum_{k=0}^\infty \int_{B_{(k)}} |u(y)| \frac{1}{1/j(R)+ 1/j(|y|/2)} dy \\
&\leq&  2j(R) \int_{B_{2R}} |u(y)|~dy+N\sum_{k=2}^\infty  j(2^{k-2}R)  \int_{B_{2^kR}} |u(y)|~dy\\
&\leq&  N\left(j(R)R^d(|u|)_{B_{2R}}+  \sum_{k=2}^\infty  2^{-(k-2)(d+\alpha_0)} j(R)  \int_{B_{2^{k}R}} |u(y)|~dy\right)\\
&\leq&  N\left(j(R)R^d(|u|)_{B_{2R}}+ \sum_{k=2}^\infty  2^{-(k-2)(d+\alpha_0)} 2^{kd}j(R)R^d   (|u|)_{B_{2^{k}R}}\right)\\
&\leq& Nj(R)R^d\left( \sum_{k=1}^\infty  2^{-\alpha_0 k}(|u|)_{B_{2^k R}} \right),
\end{eqnarray*}
where the first and second inequalities come from {\bf H1}.
Therefore we get \eqref{3212}.

To prove \eqref{3213}, we apply the operator $\cA$ to both sides of $Lu-\lambda u=f$ and obtain
$$
(L-\lambda)(\cA u) = \cA f.
$$
By applying Corollary \ref{cor1} again,
\begin{eqnarray}
\label{3215}[\cA u]_{C^\alpha(B_{R/2})} \leq N\frac{1}{j(R)R^{d+\alpha}} \left( \|\cA u\|_{L_1(\bR^d,w_R)} +\sup_{B_R}|\cA f|\right).
\end{eqnarray}
The first term on the right hand side of \eqref{3215} is bounded by
$$
NR^{-\alpha}\left( \sum_{k=0}^\infty  2^{-\alpha_0 k}(|\cA u|)_{B_{2^kR}} \right).
$$
In order to estimate the second term, we recall the definition of $\cA$. For $|x|<R$,
\begin{eqnarray*}
|\cA f(x) |
&=&  \left|  \int_{\bR^d} [f(x+y)-f(x)] J(y)~dt\right|\\
&\leq& \sum_{k=1}^\infty  \int_{B_{(k)}} |f(x+y)|j(|y|)~dy \\
&\leq& N\sum_{k=1}^\infty  j(2^{k-1}R) \int_{B_{(k)}} |f(x+y)|~dy \\
&\leq& N\sum_{k=1}^\infty  2^{-(k-1)(d+\alpha_0)}j(R) \int_{B_{2^kR}} |f(x+y)|~dy \\
&\leq& N\sum_{k=1}^\infty  2^{-(k-1)(d+\alpha_0)}j(R) \int_{B_{2^{k+1}R}} |f(y)|~dy \\
&\leq& Nj(R)R^d\left( \sum_{k=1}^\infty  2^{-\alpha_0 k}(|f|)_{B_{2^{k+1}R}} \right) \leq Nj(R)R^d \cM f(0),
\end{eqnarray*}
where the first inequality is due to the assumption $f(x)=0$ if $|x|<2R$ and both the second and the third inequality are owing to {\bf H1}.
Therefore (i) is proved. Also, (ii) is proved similarly with Corollary \ref{cor1} (ii).
\end{proof}

The above lemma easily yields the following mean oscillation estimate.

\begin{corollary}
\label{cor22}
Suppose that {\bf H1} and {\bf H2} hold.
Let $\lambda \geq 0$ an  $r, \kappa>0$. Asume $f \in C_0^\infty$,   $f=0$ in $B_{2kr}$, and  $u, \tilde u  \in H_2^{\cA} \cap C_b^\infty$ satisfy
\begin{eqnarray*}
Lu -\lambda u =f, \quad \quad
\tilde L \tilde u -\lambda \tilde u =f.
\end{eqnarray*}
(i) Then for all $\alpha \in (0, \min\{1,\alpha_0\})$,
\begin{eqnarray}
\label{5261}(|u-(u)_{B_r}|)_{B_r} \leq N \kappa^{-\alpha}\sum_{k=1}^\infty 2^{- \alpha_0 k}|u|_{B_{2^k \kappa r}},
\end{eqnarray}
\begin{eqnarray}
\label{5262}(|\cA u-\cA u)_{B_r}|)_{B_r} \leq N \kappa^{-\alpha} \left( \sum_{k=1}^\infty 2^{-\alpha_0 k }|\cA u|_{B_{2^k \kappa r}}  + \cM f(0) \right),
\end{eqnarray}
where $N$ depends only on $d,\nu,\Lambda,\kappa_1, \kappa_2,\alpha_0$, and $\alpha$.

(ii) If one of  {\bf H3} (ii)-(iv) is additionally assumed, then \eqref{5261} and \eqref{5262} hold for $\tilde u$.

\end{corollary}
\begin{proof}
It is enough to use the following inequality
\begin{eqnarray*}
(|u-(u)_{B_r}|)_{B_r} \leq 2^\alpha r^\alpha[u]_{C^\alpha(r)} \leq 2^\alpha r^\alpha[u]_{C^\alpha(\kappa r/2)}
\end{eqnarray*}
and apply Lemma \ref{holderle} with $R=\kappa r$.
\end{proof}

Next we show that the mean oscillation of $u$ is controlled by  the maximal functions of $u$ and  $Lu - \lambda u$.

\begin{lemma}
\label{le13}
Suppose that {\bf H1} and {\bf H2} hold.
Let $\lambda > 0$, $\kappa \geq 2$, $r>0$, and $f \in C_0^\infty$. Assume $u, \tilde u  \in H_2^{\cA} \cap C_b^\infty$ satisfy
\begin{equation}
                               \label{4032}
Lu -\lambda u =f, \quad \quad
\tilde Lu -\lambda u =f.
\end{equation}
(i) Then for all $\alpha \in (0, \min\{1,\alpha_0\})$,
\begin{eqnarray}
\notag &&\lambda(|u-(u)_{B_r}|)_{B_r} +(|\cA u-(\cA u)_{B_r}|)_{B_r} \\
\label{5264} &&\leq N \kappa^{-\alpha}\left( \lambda \cM u(0) + \cM (\cA u )(0) \right) + N \kappa^{d/2} (\cM(f^2)(0))^{1/2},
\end{eqnarray}
where $N$ depends only on $d,\nu,\Lambda$, and $J$.

(ii) If one of  {\bf H3} (ii)-(iv) is additionally assumed, then \eqref{5264} holds for $\tilde u$.
\end{lemma}

\begin{proof}
Due to the similarity of the proof, we only prove the assertion (i).

Take a cut-off function $\eta \in C_0^\infty(B_{4\kappa r})$ satisfying $\eta =1$ in $B_{2\kappa r}$.
By Theorem \ref{main L_2}, there exists a unique solution $u$ in $H_2^{\cA}$ satisfying
\begin{eqnarray}
 \label{4031} Lw - \lambda w = \eta f
\end{eqnarray}
and
\begin{eqnarray}
\label{4033}  \lambda \|w\|_{L_2} + \|\cA  w\|_{L_2} \leq N \|\eta f\|_{L_2}.
\end{eqnarray}
From \eqref{4033}, Jensen's inequality, and the fact $\eta f$ has its support within $B_{4\kappa r}$, for any $R>0$,
\begin{eqnarray}
\notag \lambda(|w|)_{B_R} +(|\cA w|)_{B_R}
\notag &\leq& NR^{-d/2} \left(\lambda \|w\|_{L_2} +\|\cA w\|_{L_2} \right)\\
\notag &\leq& NR^{-d/2}  \|\eta f\|_{L_2}  \\
\label{4035} &\leq& NR^{-d/2} (\kappa r)^{d/2} ( \cM(f^2)(0))^{1/2}.
\end{eqnarray}
Furthermore, taking $(1-\Delta)^{\gamma}$ to both  sides of \eqref{4031} and using the fact $(1-\Delta)^{\gamma} Lw = L (1-\Delta)^\gamma w$,
we can easily check that $w \in C_b^\infty$ by Sobolev's inequality.
By setting $v:= u-w $, from \eqref{4031} and \eqref{4032}
$$
Lv - \lambda v = (1- \eta)f, \quad v \in C_b^\infty \cap H_2^{\cA}.
$$
By applying Corollary \ref{cor22} to $v$,
\begin{eqnarray}
\notag &&(\lambda |v-(v)_{B_r}|)_{B_r}+(|\cA v-(\cA v)_{B_r}|)_{B_r} \\
\notag &\leq& N \kappa^{-\alpha} \left( \sum_{k=1}^\infty 2^{-\alpha_0 k }[\lambda(|v|)_{B_{2^k \kappa r}}+(|\cA  v|)_{B_{2^k \kappa r}}]   + \cM f(0) \right) \\
\notag &\leq& N \kappa^{-\alpha} \left( \sum_{k=1}^\infty 2^{-\alpha_0 k }[\lambda(|u|)_{B_{2^k \kappa r}}+(|\cA u|)_{B_{2^k \kappa r}}]   \right) \\
\notag && +N \kappa^{-\alpha} \left( \sum_{k=0}^\infty 2^{-\alpha_0 k }[\lambda(|w|)_{B_{2^k \kappa r}}+(|\cA w|)_{B_{2^k \kappa r}}]   + \cM f(0) \right) \\
\notag &\leq& N \kappa^{-\alpha} \left( \sum_{k=1}^\infty 2^{-\alpha_0 k }[\lambda(|u|)_{B_{2^k \kappa r}}+(|\cA u|)_{B_{2^k \kappa r}}]   \right) \\
\notag && +N \kappa^{-\alpha} \left( \sum_{k=1}^\infty 2^{-\alpha_0 k }[2^{-dk/2} ( \cM(f^2)(0))^{1/2}]   + \cM f(0) \right) \\
\label{4034} &\leq& N \kappa^{-\alpha} \left( \lambda \cM u (0) + \cM(\cA u)(0)+ (\cM (f^2)(0))^{1/2} \right),
\end{eqnarray}
where \eqref{4035} is used for the third inequality with $R=2^k \kappa r$, and  for the last inequality we use $\cM f(0) \leq (\cM (f^2)(0))^{1/2}$.
By combining \eqref{4035} and \eqref{4034},
\begin{eqnarray*}
&&\lambda(|u-(u)_{B_r}|)_{B_r} +(|\cA u-(\cA u)_{B_r}|)_{B_r} \\
&\leq& N \left( \lambda (|v-(v)_{B_r}|)_{B_r}+(|\cA v-(\cA v)_{B_r}|)_{B_r} +\lambda(|w|)_{B_r} +(|\cA w|)_{B_r} \right) \\
&\leq& N  \kappa^{-\alpha}\left( \lambda \cM u (0)\right) +  N\cM(\cA u)(0)+ N( \cM(f^2)(0))^{1/2}.
\end{eqnarray*}
Therefore, the lemma is proved.
\end{proof}

We make full use of Lemma \ref{holderle} to get the mean oscillation of $Lu$.

\begin{lemma}
                                         \label{le14}
Suppose that {\bf H1} and {\bf H2} hold.
Let $\lambda > 0$, $\kappa \geq 2$, $r>0$, and $f \in C_0^\infty$. Assume $u \in H_2^{\cA} \cap C_b^\infty$ satisfy
\begin{eqnarray*}
\cA u -\lambda u =f.
\end{eqnarray*}
Then for all $\alpha \in (0, \min\{1,\alpha_0\})$,
\begin{eqnarray*}
&&\lambda(|u-(u)_{B_r}|)_{B_r} +(|Lu-(L u)_{B_r}|)_{B_r} \\
&&\leq N \kappa^{-\alpha}\left( \lambda \cM u(0) + \cM (L u )(0) \right) + N \kappa^{d/2} (\cM(f^2)(0))^{1/2},
\end{eqnarray*}
where $N$ depends only on $d,\nu,\Lambda$, and $J$.
\end{lemma}
\begin{proof}
Exchanging the roles of  $\cA$ and $L$ in the proof of  Lemma \ref{holderle}, we easily get
\begin{eqnarray*}
[Lu]_{C^{\alpha}(B_{R/2})} \leq N R^{-\alpha} \left( \sum_{k=0}^\infty 2^{-\alpha_0 k }(Lu)_{B_{2^kR}}  + \cM f(0) \right).
\end{eqnarray*}
Therefore, the lemma is proved as we follow the proof of Lemma \ref{le13}.
\end{proof}


\section{Proof of Theorems \ref{maint} and \ref{maint2}}
                                    \label{s:5}

{\bf{Proof of Theorem \ref{maint}}}

The case $p=2$ was already proved in  Theorem \ref{main L_2}. Due to Corollary \ref{unique} and Lemmas \ref{exis}, it is sufficient to prove
\begin{equation}
                                         \label{mainin2}
          \| \cA u \|_{L_p} + \lambda \| u \|_{L_p} \leq N \|Lu - \lambda u\|_{L_p}, \quad \forall\, u \in C_0^\infty,
\end{equation}
where $N=N (d, \nu , \Lambda, \kappa_1,\kappa_2, \alpha_0)$.

 First,  assume $p >2$.
Put $f:= Lu -\lambda u$.  From Lemma \ref{le13},
for all $\alpha \in (0, \min\{1,\alpha_0\})$
\begin{eqnarray*}
&&\lambda(|u-(u)_{B_r}|)_{B_r} +(|\cA u-(\cA u)_{B_r}|)_{B_r} \\
&&\leq N \kappa^{-\alpha}\left( \lambda \cM u(0) + \cM (\cA u )(0) \right) + N \kappa^{d/2} (\cM(f^2)(0))^{1/2}.
\end{eqnarray*}
By translation, it is easy to check that the above inequality holds for all $B_r(x)$ with $x \in \bR^d$ and $r>0$.
By the arbitrariness of $r$,
\begin{eqnarray*}
&&\lambda u^\#(x) +(\cA u)^\#(x) \\
&&\leq N \kappa^{-\alpha}\left( \lambda \cM u(x) + \cM ( \cA u )(x) \right) + N \kappa^{d/2} (\cM(f^2)(x))^{1/2}.
\end{eqnarray*}
Therefore, by the Fefferman-Stein theorem  and Hardy-Littlewood maximal theorem (see, for instance, chapter 1 of \cite{Ste}), we get
\begin{eqnarray*}
\lambda \|u\|_{L_p} +\|\cA u \|_{L_p} \leq N \kappa^{-\alpha}\left( \lambda \|u\|_{L_p} + \|\cA u\|_{L_p} \right) + N \kappa^{d/2} \|f\|_{L_p}.
\end{eqnarray*}
By choosing $\kappa >2 $ large enough so that $N \kappa^{-\alpha} < 1/2$,
\begin{eqnarray*}
\lambda \|u\|_{L_p} +\|\cA u \|_{L_p} \leq N \|f\|_{L_p}.
\end{eqnarray*}

We use the duality argument for $p \in (1,2)$. Put $q := p/(p-1)$. Then since $q \in (2,\infty)$, for any $g \in C_0^\infty$ there is a unique $v_g \in H_q^\cA$ satisfying
$$
L^\ast v_g - \lambda v_g = g \quad \text{in}~ \bR^d.
$$
Therefore, by applying \eqref{mainin2} with $q \in (2,\infty)$, for any $u \in C_0^\infty$,
\begin{eqnarray*}
\|\cA u\|_{L_p}
&\leq & \sup_{ \|g\|_{L_q}=1,~ g \in C_0^\infty}\int_{\bR^d} |g\cA u |~dx\\
&=& \sup_{ \|g\|_{L_q}=1,~ g \in C_0^\infty}\int_{\bR^d} |(L^\ast v_g -\lambda v_g) \cA u |~dx\\
&=& \sup_{ \|g\|_{L_q}=1,~ g \in C_0^\infty}\int_{\bR^d} |\cA v_g (L  u -\lambda  u)|~dx\\
&\leq& \sup_{ \|g\|_{L_q}=1,~ g \in C_0^\infty} \|\cA v_g\|_{L_q} \|L  u -\lambda  u\|_{L_p}\\
&\leq& \sup_{ \|g\|_{L_q}=1,~ g \in C_0^\infty} N\|g\|_{L_q} \|L  u -\lambda  u \|_{L_p}=N \|L u -\lambda  u\|_{L_p}.
\end{eqnarray*}
Similarly,
$$
\lambda \|u\|_{L_p} \leq N \|L u -\lambda  u\|_{L_p}.
$$

Finally, we  prove the continuity of the operator $L$ by showing
\begin{eqnarray}
\label{4051}\|Lu\|_{L_p} \leq N \|\cA u\|_{p}, \quad \forall\,\, u \in C_0^\infty.
\end{eqnarray}
 Recall that we proved \eqref{mainin2} based on Lemma \ref{le13}. Similarly,  using  Lemma \ref{le14}, one can prove
\begin{eqnarray*}
\|Lu\|_{L_p} \leq N \| \cA u- \lambda u\|_{L_p} \quad \forall u \in C_0^\infty, \quad \forall \lambda >0.
\end{eqnarray*}
Since $N$ is independent of $\lambda$, this leads to (\ref{4051}). The theorem is proved.

\vspace{4mm}

{\bf{Proof of Theorem \ref{maint2}}}

The proof is identical to that of   Theorem \ref{maint} if one of {\bf{H3}}(ii)-(iv) holds.  So it only remains to prove
$$
          \| \cA u \|_{L_p} + \lambda \| u \|_{L_p} \leq N \|\tilde{L}u - \lambda u\|_{L_p}, \quad \forall\, u \in C_0^\infty
$$
 under the condition {\bf H3}(i).   Define
$$ b^i=-\int_{B_1}y^i a(y)J(y)dy \quad \text{if}\,\,\sigma\in (0,1), \quad \quad  b^i=\int_{\bR^d\setminus B_1}y^i a(y)J(y)dy \quad \text{if}\,\,\sigma\in (1,2).
$$
Then under {\bf{H1}} and {\bf H2}, $|b|<\infty$ and for
for any $u \in C_0^\infty$, we have
$$
\tilde{L}u=Lu+b\cdot \nabla u,
$$
and therefore
\begin{eqnarray*}
\|u\|_{L_p}+\|\cA u\|_{L_p}  \leq N\| Lu -\lambda u\|_{L_p} \leq N \big(\|\tilde L u - \lambda u \|_{L_p}+\|\nabla u\|_{L_p} \big).
\end{eqnarray*}
Take $\varepsilon =1/(2N)$ in {\bf H3}(i) and apply Lemma \ref{lemma 5.25.1}. Then, the theorem is proved.
\qed

\end{document}